\documentclass[11pt,reqno]{amsart}
\usepackage[T1]{fontenc}
\usepackage{lmodern}
\usepackage[a4paper,textwidth=14.8cm,textheight=22.6cm,centering]{geometry}
\usepackage{amsmath,amssymb,mathtools,enumitem,booktabs,microtype}
\usepackage{xcolor}
\usepackage[colorlinks=true,linkcolor=blue!55!black,citecolor=blue!55!black,urlcolor=blue!55!black]{hyperref}
\hypersetup{pdftitle={Sharp polynomial upper bounds for anticanonical volumes},pdfauthor={Pinxian Bie, Peien Du and Zhengjie Yu},pdfsubject={Sharp polynomial volume bounds for epsilon-lc Fano varieties},pdfkeywords={Fano varieties, anticanonical volumes, log discrepancies, boundedness}}
\numberwithin{equation}{section}
\newtheorem{theorem}{Theorem}[section]
\newtheorem{proposition}[theorem]{Proposition}
\newtheorem{lemma}[theorem]{Lemma}
\newtheorem{corollary}[theorem]{Corollary}
\theoremstyle{definition}

\theoremstyle{remark}
\newtheorem{remark}[theorem]{Remark}
\newcommand{\Q}{\mathbb Q}
\newcommand{\R}{\mathbb R}
\newcommand{\C}{\mathbb C}
\newcommand{\Z}{\mathbb Z}
\newcommand{\PP}{\mathbb P}
\newcommand{\OO}{\mathcal O}
\newcommand{\simq}{\sim_{\Q}}
\newcommand{\eps}{\epsilon}

\newcommand{\vol}{\operatorname{vol}}
\newcommand{\ord}{\operatorname{ord}}
\newcommand{\Spec}{\operatorname{Spec}}

\newcommand{\Supp}{\operatorname{Supp}}
\newcommand{\Div}{\operatorname{div}}
\newcommand{\Nklt}{\operatorname{Nklt}}
\newcommand{\trdeg}{\operatorname{trdeg}}
\newcommand{\Nm}{\operatorname{Nm}}
\newcommand{\degh}{\deg_{\mathrm h}}
\newcommand{\coef}{\operatorname{coeff}}
\newcommand{\lct}{\operatorname{lct}}

\newcommand{\cD}{\mathcal D}

\newcommand{\ba}{\mathfrak a}
\newcommand{\bb}{\mathfrak b}
\newcommand{\bj}{\mathfrak j}
\newcommand{\bM}{\mathbf M}
\setlist[enumerate]{label=\textup{(\arabic*)},leftmargin=2em,itemsep=3pt}
\title[Sharp polynomial anticanonical volume bounds]{Sharp polynomial upper bounds for anticanonical volumes}
\author{Pinxian Bie, Peien Du and Zhengjie Yu}
\address{School of Mathematical Sciences, Fudan University, Shanghai, China}
\date{October 7, 2026}

\begin{document}
\begin{abstract}
For every positive integer $n$, we prove that there is a constant $C_n$, depending only on $n$, such that
$\vol(-K_X)\le C_n\eps^{-(2^n-n-1)}$
whenever $X$ admits an $\eps$-lc log Fano boundary. The exponent is optimal in every dimension at least two, even for toric Fano varieties of Picard number one. In particular, the optimal exponent for fourfolds is eleven. We also prove an anticanonical interpolation theorem with optimal exponent $2^n-1$. The proof combines signed discrepancy estimates under projection, finite morphisms to projective space, and the canonical bundle formula. A reduction to a base bounded independently of $\eps$, followed by a volume estimate along a flag, yields the sharper volume exponent.
\end{abstract}
\maketitle
\setcounter{tocdepth}{1}
\tableofcontents

\section{Introduction}\label{sec:intro}
Throughout the paper, varieties are defined over $\C$.

Fano varieties are one of the basic objects in the birational classification of algebraic varieties. Birkar's proof of the Borisov--Alexeev--Borisov conjecture shows that, for fixed dimension $n$ and fixed $\eps>0$, Fano varieties with $\eps$-lc singularities form a bounded family \cite{Bir19,Bir21}. In particular, their anticanonical volumes are bounded. A quantitative question remains: how rapidly must a volume bound grow as the singularity parameter tends to zero?

The distinction between a bound for each fixed parameter and a polynomial bound in its inverse is substantial. The constants supplied by boundedness alone do not specify this dependence. Moreover, weighted projective spaces give examples whose volumes grow rapidly as their singularities worsen. They impose a precise candidate for the smallest possible exponent in each dimension.

In dimension three, Jiang and Zou \cite[Theorem~1.1]{JZ24} proved the sharp-order estimate $3200\eps^{-4}$ for varieties of $\eps$-Fano type, using the strict $\eps$-klt convention. Their proof refines earlier volume arguments of Jiang \cite{Jiang21}. Its main ingredients include a cubic interpolation estimate on surfaces and a suitable reduction of the base of a Mori fibre space. Birkar \cite{Bir4} obtained an explicit bound for canonical Fano fourfolds and an explicit estimate for $\eps$-lc Fano threefolds, following the covering-centre approach in \cite{Bir19}. The dependence on $\eps$ in these arguments is sensitive to both the threshold estimate and the final passage from thresholds to volumes.

We prove a polynomial bound in every dimension, with the exponent forced by the toric examples. Let $0<\eps\le1$, a variety of \emph{$\eps$-Fano type} means a normal projective $\Q$-Gorenstein variety $X$ for which there is an effective $\Q$-divisor $\Delta$ such that $(X,\Delta)$ is $\eps$-lc and $-(K_X+\Delta)$ is ample. Thus the parameter is associated to a log Fano boundary, not merely to the underlying variety.

\begin{theorem}\label{thm:main}
For every positive integer $n$ there exists $C_n>0$, depending only on $n$, such that
\begin{equation}\label{eq:main}
 \vol(-K_X)\le C_n\eps^{-a_n},\qquad a_n=2^n-n-1,
\end{equation}
for every $n$-dimensional variety $X$ of $\eps$-Fano type. In particular, the estimate holds for $\eps$-lc Fano and weak Fano varieties.

For every $n\ge2$, the exponent $a_n$ cannot be decreased, even if $X$ is required to be a $\Q$-factorial toric Fano variety of Picard number one.
\end{theorem}

The exponent is sharp; the leading constant is not optimized. The proof also applies to effective real log Fano boundaries after changing $C_n$ by a factor depending only on $n$; see Remark~\ref{rem:real}. In dimension one we can take $C_1=2$. The first exponents are
\[
 a_1=0,\quad a_2=1,\quad a_3=4,\quad a_4=11,
 \quad a_5=26,\quad a_6=57.
\]
Consequently the fourfold statement is an immediate special case:
\begin{corollary}\label{cor:fourfold}
There is an absolute constant $C_4>0$ such that every $\eps$-lc weak Fano fourfold satisfies $\vol(-K_X)\le C_4\eps^{-11}$. The exponent eleven is optimal already for Fano fourfolds of Picard number one.
\end{corollary}

To the best of our knowledge, Theorem~\ref{thm:main} gives the first polynomial upper bound with the optimal exponent for arbitrary-dimensional $\eps$-lc Fano varieties, and in particular the first such bound in dimension four. The low-dimensional exponents are consistent with the earlier results above. Sharpness comes from Ambro's family \cite[Example~6.3]{Amb16}; we give a direct weighted-projective description and a discrepancy calculation in Section~\ref{sec:examples}.

The principal threshold statement is the following.
\begin{theorem}\label{thm:interp}
For every positive integer $n$ there exists $c_n>0$, depending only on $n$, with the following property. Let $X$ be a projective $\Q$-Gorenstein variety of Fano type, and let $B,D\ge0$ be $\Q$-divisors such that
\[
 B\simq D\simq-K_X,\qquad (X,B)\text{ is }\eps\text{-lc}.
\]
Then
\begin{equation}\label{eq:introinterp}
 \bigl(X,(1-t)B+tD\bigr)\text{ is lc whenever }
 0\le t\le c_n\eps^{p_n},\qquad p_n=2^n-1.
\end{equation}
For $n\ge2$ the exponent $p_n$ cannot be replaced by a smaller exponent uniformly in this class.
\end{theorem}

The theorem compares an arbitrary effective anticanonical divisor with a specified anticanonical boundary of controlled singularities. It controls all divisorial valuations, including those contracted on subsequent models. This is needed when applying the canonical bundle formula and when moving between Mori fibre spaces.

\subsection{Outline of the proof}
Our proof has two dimension inductions. The first establishes Theorem~\ref{thm:interp}. The second estimates volumes from that interpolation theorem. Keeping them separate is essential for the exponent in \eqref{eq:main}.

The local input builds on Jiao's projection method in \cite{Jiao}. For a signed sub-pair on projective space, the loss of discrepancy under projection is expressed by finitely many quotients of coefficient ideals. The normalized generating degree of the positive ideal is bounded by the square of the positive boundary degree, independently of the negative degree. We use this construction over geometric discrete valuation rings and retain all vertical fixed divisors. A surface calculation then gives a relative discrepancy bound of order $\eps^{2^r}$ in relative dimension $r$.

For the absolute discrepancy comparison we retain the quotient by the restriction index throughout the dimension induction. This gives the recurrence
\[
 p_j=1+2p_{j-1},\qquad p_1=1,
\]
and hence $p_j=2^j-1$. This comparison is a refinement of the argument of \cite[Proposition~5.2]{Jiao}; it is not the all-valuations exponent stated in that paper. The signed realization of one selected ideal expression is what permits the refinement.

To apply projective-space estimates to terminal Fano varieties, we use a finite morphism of bounded degree. Characteristic polynomials recover the largest order among all valuation extensions, and therefore the smallest discrepancy quotient. The norm alone would only recover their sum. Fixed-parameter BAB supplies the finite morphisms, so its constants do not depend on the varying singularity parameter $\eps$.

Terminalization and the $K_X$-MMP reduce interpolation to terminal Fano fibres and generalized pairs on the base. The moduli parts are b-nef. We approximate them at a prescribed valuation by ordinary effective anticanonical divisors, with a uniform discrepancy bound for the good pair. This uses no general semiampleness statement for moduli b-divisors. If the fibre and base dimensions are $r$ and $s$, respectively, the resulting exponent is
\[
 p_r+2^rp_s=2^{r+s}-1.
\]

For volumes, a further reduction produces a polarized base whose degree is bounded independently of $\eps$. We use Birkar's theorem on the singularities of Fano fibrations \cite{BirFib}, together with the lifting result of Mauri--Moraga \cite[Lemma~2.11]{MM}. A flag on this base and connectedness give a volume bound proportional to
\[
 M_r(\eps/4)\,\mu_r(\eps/4)^{-(n-r)},
\]
where $M_r$ bounds fibre volumes and $\mu_r$ is the interpolation threshold. The largest exponent occurs for a one-dimensional base:
\[
 a_n=a_{n-1}+p_{n-1}=2^n-n-1.
\]

Section~\ref{sec:prelim} fixes the conventions and the birational inputs. Sections~\ref{sec:local} and~\ref{sec:projection} develop the signed local estimates. Section~\ref{sec:finite} passes to terminal Fano fibres, and Section~\ref{sec:interpolation} proves Theorem~\ref{thm:interp}. Sections~\ref{sec:base} and~\ref{sec:volume} prove the volume bound. The last section treats sharpness and further questions. Every reference to a numbered result of \cite{Jiao} refers to arXiv:2610.05638.

\subsection*{Acknowledgements}
The authors are grateful to their adviser Meng Chen for his great encouragement and support. They are also grateful for their fellow students for their helpful discussion and friendship.

\subsection*{AI disclosure}
The authors used ChatGPT as research-assistance tool during the preparation of this manuscript. The first author started working on this problem a year ago with some progress made along the way. Recently, the authors found some breakthrough using AI tools. In particular the threshold formula in Section 3 is suggested by ChatGPT, which is very helpful. The authors take full responsibility for the accuracy and content of the paper.

\section{Preliminaries}\label{sec:prelim}

\subsection{Conventions}
We use the standard conventions of \cite{KM98}. A contraction is a projective surjective morphism of normal varieties with connected fibres. A birational contraction is a birational map whose inverse does not contract a prime divisor. Canonical divisors on birational models are chosen compatibly.

A sub-pair $(X,B)$ consists of a normal variety and a $\Q$-divisor for which $K_X+B$ is $\Q$-Cartier. The divisor $B$ may have negative coefficients. We write $B=B^+-B^-$ for its positive and negative parts, which have no common component. We occasionally say \emph{signed divisor} to emphasize that negative coefficients are permitted. For a prime divisor $E$ on a normal birational model $g:Y\to X$, its log discrepancy is
\[
 A(E,X,B)=1-\coef_E B_Y,
 \qquad K_Y+B_Y=g^*(K_X+B).
\]
The sub-pair is sub-$\eta$-lc if $A(E,X,B)\ge\eta$ for every prime divisor over $X$. We omit ''sub'' when $B\ge0$. Interpolated boundaries may have real coefficients; discrepancies extend linearly to them.

For a $\Q$-Cartier divisor $L$ on a projective $d$-fold, we use
\[
 \vol_X(L)=\limsup_{m\to\infty}
 \frac{h^0(X,\OO_X(mL))}{m^d/d!},
\]
where $m$ is sufficiently divisible. The limsup is a limit on such multiples. A divisor is big if and only if its volume is positive, and $\vol(L)=L^d$ when $L$ is nef; see \cite[Section~2.2]{Laz04}. Intermediate anticanonical divisors in this paper are not necessarily nef, so we use volumes rather than top self-intersections.

A variety is of Fano type if it admits an effective $\Q$-boundary $\Delta$ with $(X,\Delta)$ klt and $-(K_X+\Delta)$ ample. A $\Q$-factorial variety of Fano type is a Mori dream space by \cite[Corollary~1.3.2]{BCHM}. The needed MMPs terminate. Fano type is preserved under contractions and the birational contractions considered here; see \cite[Theorem~1.1 and Corollary~1.3]{GOST} and \cite[Section~2]{MM}. We always impose the additional $\Q$-Gorenstein hypothesis when discussing the volume of $-K_X$ on a non-$\Q$-factorial variety.

\begin{lemma}\label{lem:snc}
Let $Y$ be smooth and let $C$ be a signed divisor with simple normal crossing support. If every coefficient of $C$ is at most $1-\gamma$, where $0<\gamma\le1$, then $(Y,C)$ is sub-$\gamma$-lc.
\end{lemma}
\begin{proof}
At the generic point of the centre of a divisorial valuation $v$, extend local equations for the boundary to regular parameters. The simple normal crossing discrepancy inequality gives
\[
 A(v,Y,C)\ge\sum_i(1-c_i)v(x_i),
\]
 At least one parameter has positive integral value. Every weight is at least $\gamma$, proving the assertion. In particular, negative coefficients cause no difficulty.
\end{proof}

\begin{lemma}[Standard results]\label{lem:book}
The following statements hold.
\begin{enumerate}
\item If $(X,\Delta)$ is an $\eta$-lc log Fano pair, there is $B\ge\Delta$ with $B\simq-K_X$ such that $(X,B)$ is $\eta/2$-lc. If $\eta<1$, one can retain $\eta$ instead of $\eta/2$.
\item Suppose $X\dashrightarrow X'$ is a birational contraction, $K_X+B\simq0$, and $K_{X'}+B'$ is $\Q$-Cartier, where $B'$ is the pushforward. Then the two sub-pairs are crepant.
\item If $X$ and $X'$ are $\Q$-Gorenstein and $X\dashrightarrow X'$ is a birational contraction, then $\vol(-K_X)\le\vol(-K_{X'})$. Equality holds for a small modification.
\end{enumerate}
\end{lemma}
\begin{proof}
These are standard facts in birational geometry that we shall use freely. Let $A=-(K_X+\Delta)$. For sufficiently large and divisible $m$, the system $|mA|$ is free. On a log resolution its pullback is free and has no fixed exceptional component. Choose a general member $G$ transverse to the resolved boundary and put $B=\Delta+G/m$. The old crepant coefficients are unchanged, and the new coefficient is $1/m$. Choose $1/m\le1-\eta/2$ and apply Lemma~\ref{lem:snc}. If $\eta<1$, choose $1/m\le1-\eta$ instead.

For (2), use the Negativity lemma to $p^*(K_X+B)-q^*(K_{X'}+B')$ on the common resolution $p:W\to X$ and $q:W\to X'$.

For (3), a rational section of $-mK_X$ defines an effective divisor $\Div(h)-mK_X$. Its pushforward is effective on $X'$, so the same rational function gives a section of $-mK_{X'}$. This produces an injection of spaces of sections for divisible $m$. For a small map the codimension-one conditions in the two directions are identical. Taking volumes proves the claims.
\end{proof}

\begin{lemma}[Terminalization]\label{lem:terminalization}
Let $X$ be a projective $\Q$-Gorenstein variety of Fano type. There is a projective birational morphism $\pi:Y\to X$ with $Y$ $\Q$-factorial terminal and
\begin{equation}\label{eq:terminalization}
 K_Y+E=\pi^*K_X,\qquad E\ge0.
\end{equation}
The variety $Y$ is of Fano type. If $B,D\ge0$ are anticanonical $\Q$-divisors on $X$, their crepant transforms $E+\pi^*B$ and $E+\pi^*D$ are effective. Moreover, $\vol(-K_X)\le\vol(-K_Y)$.
\end{lemma}
\begin{proof}
Since the underlying variety $X$ is klt, the terminalization and $\Q$-factorialization statements follow from the extraction results of \cite[Corollary~1.4.3]{BCHM}, extracting only divisors of log discrepancy at most one. This gives \eqref{eq:terminalization}. Pull back a klt log Fano pair on $X$ crepantly. Its boundary on $Y$ is effective, and its negative log canonical divisor is nef and big. If this divisor is denoted by $A$, write $A\simq H+G$ with $H$ ample and $G\ge0$. For sufficiently small rational $u>0$, adding $uG$ preserves klt singularities and changes the negative log canonical divisor to
\[
 A-uG\simq(1-u)A+uH,
\]
which is ample. Thus $Y$ is of Fano type. The effectiveness assertions follow from \eqref{eq:terminalization}. Finally $-K_Y=\pi^*(-K_X)+E$ gives the required inclusion of sections and the volume inequality.
\end{proof}

\subsection{Valuations and geometric discrete valuation rings}
We recall several facts about valuations and geometric DVR. Every divisorial valuation is normalized to have value group $\Z$. If $v$ is rescaled, discrepancy is rescaled by the same positive factor. For an ideal $\ba$, the number $v(\ba)$ is the minimum of the values of its local generators at the centre of $v$. For a Cartier divisor we use the order of a local equation. These definitions do not depend on the choices.

A \emph{geometric discrete valuation ring}, or geometric DVR, is the local ring at a prime divisor on a normal model of a function field over $\C$. It is regular; we may take the model smooth near that prime divisor. Its maximal ideal is generated by a \emph{uniformizer} $u$. Its fraction field is denoted by $K$, and its residue field by $\kappa$. The latter need not be algebraically closed.

\begin{lemma}\label{lem:restriction}
Let $L/K$ be a finitely generated extension of function fields over $\C$, of transcendence degree $r$. If a normalized divisorial valuation $v$ of $L$ restricts nontrivially to $K$, then
\[
 v|_K=e\zeta
\]
for an integer $e\ge1$ and a normalized divisorial valuation $\zeta$ of $K$.
\end{lemma}
\begin{proof}
The nonzero subgroup $v(K^*)\subseteq\Z$ equals $e\Z$; dividing by $e$ gives a discrete valuation $\zeta$. Residue classes algebraically independent over $\kappa(\zeta)$ lift to elements algebraically independent over $K$: in any proposed polynomial relation, divide the coefficients by one of least value and reduce modulo the valuation ideal. Hence
\[
 \trdeg\bigl(\kappa(v)/\kappa(\zeta)\bigr)\le r.
\]
Since $v$ is divisorial, this implies $\trdeg_\C\kappa(\zeta)\ge\trdeg_\C K-1$. Conversely, an element of positive value and lifts of an algebraically independent set in the residue field are algebraically independent over $\C$, by comparing the least powers of that element. This proves the opposite inequality. The discrete valuation $\zeta$ therefore has residue transcendence degree $\trdeg_\C K-1$, which is the divisorial criterion.
\end{proof}

For a finite extension of degree $N$, the restriction index $e$ is at most $N$. When passing to a splitting field we extend the original valuation without normalizing it again. Thus, if a splitting DVR has uniformizer $s$ and $\ord_s(u)=h$, its valuation divided by $h$ assigns value one to $u$.

We will repeatedly localize a morphism at a geometric DVR of the base. The boundary on the resulting model is always the crepant boundary, including the canonical correction from the base. Concretely, wedge a relative top form with differentials of a separating transcendence basis of the residue field. This identifies the discrepancy calculation on the localized scheme with the discrepancy on the original variety. Dropping a vertical coefficient during this localization would change that discrepancy.

\begin{lemma}[Finite ramification formula]\label{lem:ramification}
Let $\pi:T\to Y$ be a finite surjective morphism of normal varieties in characteristic zero, with $Y$ smooth. Let $(T,B)$ be a sub-pair and set
\[
 H=K_T-\pi^*K_Y+B.
\]
If $v|_{K(Y)}=e\zeta$ for normalized divisorial valuations, then
\begin{equation}\label{eq:finite-ram}
 A(v,T,B)=eA(\zeta,Y,0)-v(H).
\end{equation}
The formula also holds on the localized schemes used later.
\end{lemma}
\begin{proof}
Represent $\zeta$ on a smooth model of $Y$, normalize that model in $K(T)$, and work at the prime representing $v$. For an extension of geometric DVRs, write $u=a s^e$ with $a$ a unit. After choosing separating residue parameters, differentiation gives
\[
 du=s^{e-1}(ea+s\partial_s a)\,ds.
\]
The parenthesis is a unit in residue characteristic zero. Thus the relative canonical coefficient is $e-1$. The same calculation with top forms proves that log discrepancy is multiplied by $e$ under a finite crepant pullback. Subtracting $v(H)$ yields \eqref{eq:finite-ram}. This proof only requires $K_T+B$ to be $\Q$-Cartier, not $K_T$ separately.
\end{proof}

\subsection{Degree and discrepancy on projective space}
For a signed divisor $C$ on $\PP^j$ its degree is $C\cdot\OO(1)^{j-1}$. On $\PP^r_R$, its \emph{horizontal degree} is the degree of its restriction to $\PP^r_K$. The special fibre has horizontal degree zero. An ideal of horizontal generating degree $d$ means, in this paper, an ideal $\ba$ for which $\ba\otimes\OO(d)$ is generated by finitely many global sections over $\Spec R$. No bound on their orders in $u$ is implicit in this terminology.

\begin{lemma}\label{lem:degree}
For an effective $\Q$-divisor $T$ on $\PP^j$ and a normalized divisorial valuation $v$, one has
\[
 v(T)\le (\deg T)A(v,\PP^j,0).
\]
\end{lemma}
\begin{proof}
It suffices to treat a hypersurface of degree $d$. Choose an affine chart meeting the generic point of the centre and let $f$ be its polynomial equation. On a smooth model representing $v$, the Jacobian determinant of the birational map has order $A(v,\PP^j,0)-1$. The inverse Jacobian formula therefore shows that differentiating a regular polynomial in an affine coordinate lowers its valuation by at most $A(v,\PP^j,0)$. Some iterated derivative of $f$ of order at most $d$ is a nonzero constant. Iterating the inequality proves $v(f)\le dA(v,\PP^j,0)$. Summing with the positive coefficients of $T$ proves the assertion. This is also the argument of \cite[Lemma~2.12]{Jiao}.
\end{proof}

\subsection{Generalized pairs and the canonical bundle formula}
A $\Q$-b-divisor on $Z$ associates a $\Q$-Weil divisor to each normal birational model, compatibly with pushforward. It is $\Q$-b-Cartier and b-nef if, on some model $p:U\to Z$, it is determined by a nef $\Q$-Cartier divisor $M_U$ and by its pullbacks to higher models. A generalized sub-pair $(Z,B+\bM)$ consists of such data with $K_Z+B+M_Z$ $\Q$-Cartier. Write
\[
 K_U+B_U+M_U=p^*(K_Z+B+M_Z).
\]
The generalized log discrepancy of a prime $E$ on $U$ is $1-\coef_E B_U$. Generalized $\delta$-lc means this is at least $\delta$ on every model. We allow coefficients of $B$ above one when explicitly discussing a generalized sub-pair with no singularity assumption.

We use the canonical bundle formula in the klt-trivial form of Ambro \cite[Theorem~2.7]{Amb04}; see also \cite[Definition~3.1 and Theorem~3.5]{FG14}. If $f:W\to Z$ is a contraction, $(W,C)$ is sub-klt over the generic point, $K_W+C\simq f^*L$, and the rank-one condition holds, then
\[
 K_Z+B_Z+M_Z\simq L,
\]
with $\bM$ $\Q$-b-Cartier and b-nef. The coefficient of the discriminant at a prime $P$ on any base model is $1-a_C(P)$, where
\begin{equation}\label{eq:discrdef}
 a_C(\zeta)=\inf_{v|_{K(Z)}=e\zeta}\frac{A(v,W,C)}{e}.
\end{equation}
The expression can be negative if the total pair is not lc over that prime. Generic sub-klt singularities, rather than global lc singularities, suffice for this version of the formula.

In our applications $C\ge0$ and the generic-fibre pair is klt. The rank-one condition follows as follows. On a resolution of the generic fibre, the round-up of the discrepancy divisor has only effective exceptional contributions: coefficients along the strict transform of $C$ lie in $(-1,0]$. Its pushforward is the structure sheaf. The connected generic fibre then has only the ground field as global functions. This is precisely the required rank one. Also $B_Z\ge0$: over the generic point of a prime $P\subset Z$, a component of $f^*P$ has integral multiplicity $m\ge1$ and boundary coefficient $b\ge0$, giving $a_C(P)\le(1-b)/m\le1$.

\section{Signed thresholds over a discrete valuation ring}\label{sec:local}

We first establish the two-dimensional estimate that will be iterated. All the statements in this section hold over a geometric DVR $R$ with uniformizer $u$, without an algebraic-closedness assumption on its residue field. Horizontal means dominating $\Spec R$; vertical means supported over its closed point. A vertical valuation is one for which $v(u)>0$.

\subsection{A surface estimate}
\begin{proposition}\label{prop:surface}
Let $(\PP^1_R,C)$ be a signed sub-$\eta$-lc pair, with $0<\eta\le1$ and $\degh C^+\le k$, where $k\ge1$. For every normalized vertical divisorial valuation,
\begin{equation}\label{eq:surface}
 A(v,\PP^1_R,C)\ge
 \frac{\eta}{2\lceil2k/\eta\rceil}v(u)
 \ge\frac{\eta^2}{6k}v(u).
\end{equation}
\end{proposition}
\begin{proof}
Successively blow up closed points of the special fibre of this regular surface. Its reduced vertical support stays simple normal crossing. For a vertical component $D$ write
\[
 m_D=\ord_D(u),\qquad a_D=A(D,\PP^1_R,C).
\]
Let $P$ be the next centre and $E$ the new exceptional curve. Write $\mu$ for the signed multiplicity at $P$ of the strict horizontal boundary, and $\mu^+$ for the multiplicity of its positive part. If $P$ lies on one or two vertical components, the blowup formula gives, respectively,
\begin{equation}\label{eq:blowup}
 (m_E,a_E)=
 \begin{cases}
 (m_D,a_D+1-\mu),&P\in D\text{ only},\\
 (m_D+m_{D'},a_D+a_{D'}-\mu),&P\in D\cap D'.
 \end{cases}
\end{equation}
All coefficients in these formulas are signed crepant coefficients.

Each horizontal curve is proper and finite over $R$ and is flat, since it is torsion free over the DVR. On each blown-up model its intersection with the total special fibre equals its generic degree. At $P$ the local contribution of the positive horizontal boundary is at least
\[
 [\kappa(P):\kappa],\Bigl(\sum_{D\ni P}m_D\Bigr)\mu^+.
\]
Indeed, each vertical component is regular at $P$, and its local intersection with a horizontal curve is at least the multiplicity of that curve. Thus
\begin{equation}\label{eq:horizontal-budget}
 \mu\le\mu^+\le k/m_E.
\end{equation}
This also explains why non-rational closed points cause no loss.

Set $M=\lceil2k/\eta\rceil$ and $d=\eta/(2M)$. We prove inductively that
\begin{equation}\label{eq:surfaceind}
 a_D\ge dm_D,\qquad
 a_D\ge dm_D+\eta/2\quad\text{if }m_D>M.
\end{equation}
When $m_D\le M$, the hypothesis $a_D\ge\eta$ proves the first assertion. Suppose $m_E>M$; then \eqref{eq:horizontal-budget} gives $\mu\le\eta/2$. In the one-component case the parent also has multiplicity greater than $M$, and $1-\mu\ge\eta/2$, so the stronger assertion follows. In the two-component case there are three possibilities. If both parents have multiplicity at most $M$, then $a_E\ge2\eta-\eta/2$ and $dm_E\le\eta$. If exactly one parent has larger multiplicity, use the extra $\eta/2$ for that parent, the bound $a_{D'}\ge\eta$ for the other, and $dm_{D'}\le\eta/2$. If both are larger, the two extra terms pay for $\mu$ and leave $\eta/2$. This proves \eqref{eq:surfaceind} in every case.

Every vertical divisorial valuation on a regular surface is extracted by a finite sequence of point blowups. Applying \eqref{eq:surfaceind} to its representing component proves the first inequality in \eqref{eq:surface}. Finally $M\le3k/\eta$. The calculation is the signed version of the surface argument in \cite[Proposition~3.1]{Jiao}; the signed multiplicity and the total, possibly nonreduced, fibre are both essential.
\end{proof}

\subsection{Contacts of roots and the threshold formula}
Let $H=\sum_i b_iH_i$ be a signed horizontal divisor on $\PP^1_R$, disjoint from the infinity section, with $b_i\le1$. Each component has a monic equation in the affine coordinate $x$. In a finite splitting extension write its roots as distinct elements $\alpha_j$, assigning to a root the coefficient of its component. With the extended valuation normalized by $\ord(u)=1$, put
\[
 c_{ij}=\ord(\alpha_i-\alpha_j),\qquad c_{ii}=+\infty.
\]
These numbers are the \emph{contact orders}. All finite contact orders are nonnegative rational numbers. Define
\begin{equation}\label{eq:contactI}
 I(H)=\max\left\{0,\ \sum_i b_i\min(t,c_{ji})-t:
 j\text{ a root},\ t\ge0\right\}.
\end{equation}
The maximum is finite and is attained: each function is piecewise linear, and its final slope is $b_j-1\le0$.

\begin{lemma}[Root formula; cf.\ {\cite[Proposition~4.1]{Jiao}}]\label{lem:root}
With this notation,
\begin{equation}\label{eq:rootthreshold}
 \inf_{v\ \mathrm{vertical}}
 \frac{A(v,\PP^1_R,H)}{v(u)}=1-I(H).
\end{equation}
The infimum is attained and is unchanged by adding a nonpositive multiple of the infinity section.
\end{lemma}
\begin{proof}
Take a splitting DVR $R'$ with uniformizer $s$ and $\ord_s(u)=h$. Given a normalized valuation $v$ downstairs, extend it to a normalized valuation $w$ upstairs and write
\[
 w|_{K(x)}=av,\quad e=v(u),\quad e'=w(s),
 \qquad ae=he'.
\]
The finite canonical correction has coefficient $h-1$ on the special fibre. Lemma~\ref{lem:ramification} gives
\begin{equation}\label{eq:rootram}
 aA(v,\PP^1_R,H)
 =A(w,\PP^1_{R'},H')+(h-1)e'.
\end{equation}

In the finite chart select $j$ maximizing $\rho=w(x-\alpha_j)/(ae)$. For every $i$,
\begin{equation}\label{eq:exactcontact}
 \frac{w(x-\alpha_i)}{ae}=\min(\rho,c_{ji}).
\end{equation}
For unequal candidate orders this is the valuation rule for a sum. When they are equal, that rule gives the lower bound and maximality of $\rho$ gives the upper bound. Exact equality matters because some coefficients $b_i$ may be negative. The parameter discrepancy inequality for $(s,x-\alpha_j)$ gives
\[
 A(w,\PP^1_{R'},0)\ge e'+ae\rho.
\]
Substitute this and \eqref{eq:exactcontact} into \eqref{eq:rootram}. Using $he'=ae$, we obtain
\[
 \frac{A(v,\PP^1_R,H)}e
 \ge1+\rho-\sum_i b_i\min(\rho,c_{ji})\ge1-I(H).
\]
At infinity the horizontal order is zero and the smooth special fibre gives a quotient at least one.

Conversely, if a positive rational $t$ realizes the maximum, choose the monomial valuation in $(s,x-\alpha_j)$ whose ratio of weights is $ht$. Such a valuation is divisorial: the Euclidean sequence of point blowups extracts it. The parameter discrepancy inequality is then equality, and the orders in \eqref{eq:exactcontact} are given by the monomial minimum rule. Restrict this valuation to $K(x)$ and normalize. Equation~\eqref{eq:rootram} gives exactly $1-I(H)$. If $I(H)=0$, the original special-fibre valuation attains one. The minimizing valuation is in the finite chart, where a non-positive infinity section has order zero. This proves the last assertion as well.
\end{proof}

\subsection{Two identities for polynomial coefficients}
We record the algebraic identities used to turn \eqref{eq:contactI} into ideals. A \emph{Gauss extension} of a valuation to a rational function field in independent variables assigns to a polynomial the minimum of the values of its coefficients, and to a quotient the difference of these minima.

\begin{lemma}\label{lem:polynomial}
Let $z^N+c_1z^{N-1}+\cdots+c_N$ be a monic polynomial over a valued field, with $c_N\ne0$ and $c_0=1$. Extend the valuation to a splitting field, with no further rescaling. If its root orders are $\lambda_1\le\cdots\le\lambda_N$, then
\begin{equation}\label{eq:maxroot}
 \lambda_N=\max_{0\le j<N, c_j\ne0}
 \frac{\ord(c_N)-\ord(c_j)}{N-j}.
\end{equation}
For any finite list of eigenvalues $\xi_i$, some nonzero, and their elementary symmetric functions $e_j$, one also has
\begin{equation}\label{eq:minroot}
 \min_{\xi_i\ne0}\ord(\xi_i)
 =\min_{j\ge1, e_j\ne0}\frac{\ord(e_j)}j.
\end{equation}
\end{lemma}
\begin{proof}
Note that every term of $c_j$ has order at least $\lambda_1+\cdots+\lambda_j$. Hence each quotient in \eqref{eq:maxroot} is at most $\lambda_N$. If precisely $a$ roots have largest order, the product of all other roots is the unique least-order term in $c_{N-a}$. Its order is the sum of the $N-a$ smaller orders, so that quotient is $\lambda_N$. This includes $a=N$, where $c_0=1$.

For \eqref{eq:minroot}, every product in $e_j$ has order at least $j$ times the minimum. If precisely $q$ eigenvalues have the minimum finite order, their product is the unique least-order term of $e_q$. Both arguments allow repetitions and cancellation in other coefficients, and the second allows zero eigenvalues.
\end{proof}

\subsection{Finite lists and multiplicity at a maximum}
Choose an integer $m>0$ clearing all the coefficients of $H$. Make $mb_i$ copies of each positive root and $m|b_i|$ copies of each negative root. These copies are called occurrences. Let their total counts be
\[
 P=m\degh H^+,\qquad N=m\degh H^-.
\]
A positive root has at most $m$ occurrences. For a positive centre $j$, let $A_j(q)$ be the sum of the $q$ smallest contacts with positive occurrences, and $B_j(\ell)$ the sum of the $\ell$ smallest contacts with negative occurrences. Empty sums are zero. The negative contacts are all finite.

\begin{lemma}[Finite contact data; cf.\ {\cite[Lemma~4.2]{Jiao}}]\label{lem:lists}
If $P\le m$, then $I(H)=0$. If $P>m$, put
\[
 q_\ell=P-m-N+\ell,qquad
 \max(0,m+N-P)\le\ell\le N.
\]
Then
\begin{equation}\label{eq:listidentity}
 mI(H)=\max\{0,A_j(q_\ell)-B_j(\ell)\}_{j,\ell},
\end{equation}
where $j$ ranges over positive centres. If $I(H)>0$, for some fixed $\ell$ the maximum in \eqref{eq:listidentity} is attained by more than $m$ positive occurrences simultaneously.
\end{lemma}
\begin{proof}
The first assertion follows because the total positive coefficient is at most one. For the second, let $a_1\le\cdots\le a_P$ and $b_1\le\cdots\le b_N$ be the two lists for a centre. Directly from their ordering,
\[
 \sum_{i=1}^N\min(t,b_i)
 =\min_{0\le\ell\le N}\{B_j(\ell)+(N-\ell)t\},
\]
and, for $0\le q\le P-m$,
\[
 \max_{t\ge0}\left\{\sum_{i=1}^P\min(t,a_i)-(P-q)t\right\}
 =A_j(q).
\]
For the latter identity, the first $q$ terms are at most their contact values and each remaining term $\min(t,a_i)-t$ is nonpositive. Equality holds for $t$ between the $q$-th and $(q+1)$-st contacts. There are at most $m$ infinite entries, so the required prefix sum is finite. Substituting the first identity into $mI(H)$ and then maximizing in $t$ gives the stated range of $q_\ell$; values $q_\ell<0$ cannot give a positive maximum.

To make the results more self-contained, we justify restriction to positive centres and the multiplicity assertion. If $I(H)>0$, let $t_0$ be the earliest contact level at which the global maximum is attained. In the contact cluster containing a maximizing centre, let $P_0,N_0$ count positive and negative occurrences. A contact cluster at level $t_0$ means the roots whose mutual contacts are at least $t_0$; the ultrametric inequality makes this an equivalence class. The slope immediately before the first maximum is positive, so
\[
 P_0-N_0>m.
\]
Every strict subcluster immediately after that level has signed count at most $m$, since otherwise a centre in it would increase beyond the global maximum. Every centre in the original cluster has the same value at $t_0$, by the ultra-metric inequality. In particular, positive centres exist there and give that maximum.

Choose $\ell_0=N-N_0$. Then
\[
 q_{\ell_0}=(P-P_0)+(P_0-N_0-m).
\]
At each positive centre in the cluster, the $P-P_0$ outside positive occurrences have contacts less than $t_0$. At least $P_0-N_0-m$ additional positive contacts equal $t_0$: the strict subcluster containing the centre has positive count at most $m+N_0$. The $\ell_0$ smallest negative contacts are exactly those outside the original cluster. These observations give
\[
 A_j(q_{\ell_0})-B_j(\ell_0)=mI(H)
\]
at all $P_0>m$ positive occurrences in that cluster. This proves both remaining assertions.
\end{proof}

\section{Coefficient ideals and discrepancy under projection}\label{sec:projection}

The purpose of this section is to preserve the discrepancy quotient while lowering the dimension. We work both on $\PP^j$ over $\C$ and on $\PP^r_R$. The absolute coefficient-ideal construction is due to Jiao \cite[Proposition~4.3]{Jiao}. Here we give its algebraic construction in the relative setting, where vertical orders are unrestricted.

\subsection{Projection and its coefficient ideals}
Write $Y=\PP^r_R$, $Y_0=(u=0)$, and $\Theta=H+cY_0$ with $H$ horizontal. Each prime component of $H$ has a primitive homogeneous equation: its coefficients are not all divisible by $u$. Choose a point of $\PP^r_\kappa$ outside the reductions of all these equations and lift it to an $R$-section $q$. Such a point exists because $\kappa$ is infinite. After a linear change of coordinates the equations are monic in the coordinate of projection from $q$.

Blowing up $q$ resolves projection as a $\PP^1$-bundle over $Z=\PP^{r-1}_R$. The exceptional section has crepant coefficient $1-r\le0$, and is the section at infinity. The horizontal components lie in its complement, the total space of $\OO_Z(1)$, and are finite flat over $Z$. For a normalized divisorial valuation $\zeta$ on $Z$ define
\begin{equation}\label{eq:athreshold}
 a_\Theta(\zeta)=\inf_{v|_{K(Z)}=e\zeta}
 \frac{A(v,Y,\Theta)}e.
\end{equation}
By localizing over $\zeta$ and using Lemma~\ref{lem:root},
\begin{equation}\label{eq:rootbase}
 a_\Theta(\zeta)=A(\zeta,Z,cZ_0)-I_H(\zeta).
\end{equation}
Here $I_H$ is computed from the projected horizontal equations. The base Jacobian replaces the constant one in \eqref{eq:rootthreshold} by $A(\zeta,Z,0)$, and the term $cY_0$ subtracts $c\zeta(u)$. The non-positive infinity section does not alter the infimum. The same formula holds absolutely, with $Y=\PP^j$, $Z=\PP^{j-1}$, and $cZ_0$ omitted.

\begin{proposition}\label{prop:ideals}
Assume the coefficients of $H$ are at most one, and put $s=\degh H^+$. If $s\le1$, then $I_H=0$. If $s>1$, there are finitely many triples $(\ba,\bb,d)$ of nonzero coherent ideals on $Z$ and positive integers, with generating twists $d_+,d_-\ge0$, such that
\begin{equation}\label{eq:idealformula}
 I_H(\zeta)=\max\left\{0,
 \frac{\zeta(\ba)-\zeta(\bb)}d\right\}_{(\ba,\bb,d)},
 \qquad \frac{d_+}d\le s^2.
\end{equation}
The ideals $\ba(d_+)$ and $\bb(d_-)$ are globally generated over $R$. The identity holds simultaneously for every normalized divisorial valuation with centre on $Z$. There is no bound on $\degh H^-$ or on the vertical orders of these ideals.
\end{proposition}
\begin{proof}
If $s\le1$, we shall apply Lemma~\ref{lem:lists}. Otherwise clear the coefficients by $m$ as in that lemma. Let $\cD_+\to Z$ be the disjoint union of $mb_i$ copies of each positive horizontal component, and define $\cD_-\to Z$ similarly for negative components. These are copies of finite flat schemes, not nonreduced multiples of divisors. Their ranks are $P=ms$ and $N=m\degh H^-$. Monicity supplies a basis $1,x,\ldots,x^{d_i-1}$ for each component algebra over $Z$, so flatness is valid even above nonreduced special fibres.

On $\cD_+\times_Z\cD_\pm$, the difference of the two fibre coordinates is a section twisted by $\OO_Z(1)$. Multiplication by its square is an endomorphism of the corresponding finite locally free algebra, twisted by $\OO_Z(2)$. For an admissible $\ell$, take the $q_\ell$-th exterior power of the positive map and the $\ell$-th exterior power of the negative map. On a split fibre over a chosen root $j$, their eigenvalues are the products of $q_\ell$ positive or $\ell$ negative squared contact differences. Occurrences are selected separately even when they correspond to the same root. The zeroth exterior power is the identity, which also covers $N=0$.

Write $Q_a^+$ and $Q_a^-$ for their characteristic coefficients of positive index $a$. They are sections of the pullbacks of $\OO_Z(2q_\ell a)$ and $\OO_Z(2\ell a)$. Choose $L$ divisible by all these indices and form
\[
 U_\ell=\sum_a T_a^+(Q_a^+)^{L/a},\qquad
 V_\ell=\sum_a T_a^-(Q_a^-)^{L/a},
\]
using independent formal variables for the two expressions and for different $\ell$. These have twists $2Lq_\ell$ and $2L\ell$. Equation~\eqref{eq:minroot}, applied to the Gauss valuation in the formal variables, gives at a root $j$
\begin{equation}\label{eq:gaussorders}
 \ord_G U_{j\ell}=2LA_j(q_\ell),\qquad
 \ord_G V_{j\ell}=2LB_j(\ell).
\end{equation}
Both expressions are nonzero at every positive generic root: $q_\ell\le P-m$ excludes a forced zero product, and positive and negative components have no common root. The formal variables are never specialized. They ensure that a single collection of coefficients records the minimum for every valuation.

Now take the finite flat norm to $Z$:
\begin{equation}\label{eq:normcoeff}
 \Nm_{\cD_+/Z}(V_\ell z-U_\ell)
 =\sum_{h=0}^{P} C_{\ell h}(T)z^{P-h}.
\end{equation}
It is the determinant of multiplication on an algebra of rank $P$, and is therefore defined over all of $Z$. Its leading and constant coefficients are nonzero. Each coefficient in the formal variables in $C_{\ell h}$ is a section of
\begin{equation}\label{eq:twist}
 \OO_Z(d_{\ell h}),\qquad
 d_{\ell h}=2L\{hq_\ell+(P-h)\ell\}.
\end{equation}
This follows because a term chooses $h$ factors of $U$ and $P-h$ factors of $V$. Let $\bj_{\ell h}$ be the base ideal generated by these coefficient sections, omitting identically zero coefficients. Then $\bj_{\ell h}(d_{\ell h})$ is globally generated, and
\[
 \zeta(\bj_{\ell h})=\ord_G C_{\ell h}.
\]
This equality uses the actual coefficient ideal, so it retains any common power of $u$.

After dividing \eqref{eq:normcoeff} by its leading coefficient, its roots are $U_{j\ell}/V_{j\ell}$, counted with their occurrence multiplicities. Lemmas~\ref{lem:polynomial} and~\ref{lem:lists} imply
\begin{equation}\label{eq:retained}
 I_H(\zeta)=\max\left\{0,
 \frac{\zeta(\bj_{\ell P})-\zeta(\bj_{\ell h})}
 {2Lm(P-h)}\right\}_{\ell,\ 0\le h\le P-m}.
\end{equation}
Indeed, every quotient is at most $I_H(\zeta)$. If the maximum is positive, one $\ell$ attains the largest root order at more than $m$ occurrences. The coefficient providing equality in \eqref{eq:maxroot} then has $h<P-m$, hence is retained. When $I_H=0$, all root orders in question are nonpositive, and so are the retained quotients.

Take $\ba=\bj_{\ell P}$, $\bb=\bj_{\ell h}$ and $d=2Lm(P-h)$. The positive twist is $d_+=2LPq_\ell$. Since $P-h\ge m$ and $q_\ell\le P-m\le P$,
\[
 \frac{d_+}d=\frac{Pq_\ell}{m(P-h)}\le\frac{P^2}{m^2}=s^2.
\]
All operations were exterior powers, characteristic coefficients, determinants, powers, or extraction of coefficients over $R$. None divides out a vertical common factor. This proves the asserted relative extension of \cite[Proposition~4.3]{Jiao}.
\end{proof}

\subsection{Realizing one ideal expression by a divisor}
\begin{lemma}\label{lem:realize}
Let $Z$ be projective space over $\C$ or over a geometric DVR, and let $F$ be a signed divisor. Suppose a valuative function is the minimum of finitely many expressions of the form
\[
 A(\zeta,Z,F),\qquad
 A(\zeta,Z,F)-\alpha\zeta(\ba)+\alpha\zeta(\bb),
\]
where $\alpha>0$ is rational and the ideals have globally generated twists $d_+,d_-$. Assume every expression is at least $\gamma$, for all normalized divisorial valuations, where $0<\gamma<1$. For a prescribed $\zeta_0$, there is a signed sub-$\gamma$-lc divisor $C$ on $Z$ realizing the minimum at $\zeta_0$. Its positive degree is at most $\deg F^++\alpha d_+$, with horizontal degree in the relative case.
\end{lemma}
\begin{proof}
We shall choose an expression giving the minimum at $\zeta_0$. Resolve the two ideals, the support of $F$, and the valuation $\zeta_0$ on $p:U\to Z$. Write
\[
 \ba\OO_U=\OO_U(-G_+),\qquad
 \bb\OO_U=\OO_U(-G_-).
\]
The signed divisor $C_U^0$ defined by
\[
 K_U+C_U^0=p^*(K_Z+F)+\alpha G_+-\alpha G_-
\]
has coefficients at most $1-\gamma$, by the assumed bound for this expression. The support may be made simple normal crossing. The residual systems of $p^*\ba(d_+)$ and $p^*\bb(d_-)$ are free. Choose $q$ general positive members, each with coefficient $\alpha/q\le1-\gamma$, and one general negative member with coefficient $-\alpha$. They may be chosen transverse to the fixed support and not containing the prime representing $\zeta_0$. Adding them to $C_U^0$ gives a simple normal crossing sub-$\gamma$-lc divisor by Lemma~\ref{lem:snc}. It is the crepant pullback of the corresponding signed combination $C$ of generating divisors on $Z$. The residual members have order zero at $\zeta_0$, so the prescribed equality holds. The positive members have total weighted degree $\alpha d_+$; negative members cannot increase the positive part beyond the stated bound.

Over $R$, spread the resolution and the finite systems to a neighbourhood of the generic point of the base divisor on a complex variety. After shrinking, they remain free. Bertini over $\C$ gives the same choices. A system of twist zero may still have a vertical fixed divisor; it remains in $G_+$ or $G_-$ throughout. If the minimum is the expression without ideals, use the same argument with trivial ideals.
\end{proof}

This lemma only realizes the function at one chosen valuation. No upper bound for its value is imposed, and no effective extraction with that value is required. Its global sub-$\gamma$-lc conclusion holds for the chosen representative, which is enough for the induction below.

\subsection{The relative projective-space estimate}
\begin{theorem}\label{thm:relative}
Let $(\PP^r_R,\Theta)$ be a signed sub-$\eta$-lc pair, with $0<\eta\le1$ and $\degh\Theta^+\le k$, $k\ge1$. Every normalized vertical divisorial valuation satisfies
\begin{equation}\label{eq:relative}
 A(v,\PP^r_R,\Theta)\ge
 \frac{\eta^{2^r}}{6^{2^r-1}k^{r2^{r-1}}}\,v(u).
\end{equation}
\end{theorem}
\begin{proof}
For $r=1$ this is Proposition~\ref{prop:surface}. Suppose $r\ge2$. Use the projection above, writing $\Theta=H+cY_0$. At every geometric DVR of $Z$, the crepant localized pair on the $\PP^1$-bundle is sub-$\eta$-lc. Its positive horizontal degree is at most $k$; the infinity section is nonpositive. Proposition~\ref{prop:surface} gives
\[
 a_\Theta(\zeta)\ge\gamma:=\eta^2/(6k)
\]
for all normalized divisorial valuations $\zeta$ of the base. By \eqref{eq:rootbase} and Proposition~\ref{prop:ideals}, each of the finitely many ideal expressions for $a_\Theta$ is therefore at least $\gamma$ globally.

Given $v$, its restriction is nontrivial because $v(u)>0$. Write it as $e\zeta$ using Lemma~\ref{lem:restriction}. Lemma~\ref{lem:realize} produces a signed sub-$\gamma$-lc divisor $C$ on $Z$ with $\degh C^+\le k^2$ and
\[
 A(\zeta,Z,C)=a_\Theta(\zeta).
\]
Its vertical part includes $cZ_0$ and the fixed divisors of the selected ideals. Apply the induction hypothesis:
\[
 a_\Theta(\zeta)\ge
 \frac{(\eta^2/(6k))^{2^{r-1}}}
 {6^{2^{r-1}-1}(k^2)^{(r-1)2^{r-2}}}\,\zeta(u)
 =\frac{\eta^{2^r}}{6^{2^r-1}k^{r2^{r-1}}}\,\zeta(u).
\]
Finally, $A(v,\PP^r_R,\Theta)\ge e a_\Theta(\zeta)$ and $v(u)=e\zeta(u)$. Multiplication by the same restriction index on both sides proves \eqref{eq:relative}.
\end{proof}

\subsection{An absolute comparison with exponent $2^j-1$}
Define
\begin{equation}\label{eq:absolconstants}
 p_j=2^j-1,\quad b_j=j2^{j-1}-1,\quad
 A_1=1,\qquad A_j=2+6^{p_{j-1}}A_{j-1}\quad(j\ge2).
\end{equation}
\begin{theorem}\label{thm:absolute}
If $(\PP^j,\Theta)$ is signed sub-$\eta$-lc and $\deg\Theta^+\le k$, with $k\ge1$ and $0<\eta\le1$, then
\begin{equation}\label{eq:absolute}
 A(v,\PP^j,0)\le A_j k^{b_j}\eta^{-p_j}
 A(v,\PP^j,\Theta)
\end{equation}
for every normalized divisorial valuation $v$. In dimensions two and three, the factors $8k^3\eta^{-3}$ and $1730k^{11}\eta^{-7}$ suffice.
\end{theorem}
\begin{proof}
In dimension one every divisorial valuation is a point valuation and has smooth discrepancy one. The assertion follows from sub-$\eta$-lc. In any dimension a valuation whose centre is a prime divisor is handled in the same way. If $\Theta^+=0$, the comparison is immediate. We may thus assume a centre of codimension at least two and a nonzero positive part.

Project from a point $q$ outside the centre and outside $\Supp\Theta$. The restriction of $v$ to the function field of the base is nontrivial: its centre projects into a proper subvariety of the base. Write it as $e\zeta$, and put $a_0(\zeta)=A(\zeta,\PP^{j-1},0)$. We first prove
\begin{equation}\label{eq:comparison}
 \eta A(v,\PP^j,0)
 \le(1+\eta)A(v,\PP^j,\Theta)
 +(k^2-1)e a_0(\zeta).
\end{equation}
This is the signed form of the calculation in \cite[Proposition~5.2]{Jiao}.

Write $\Theta^+=\sum_i b_iD_i$ with monic projected equations $G_i$ of degrees $d_i$. Their discriminants and pairwise resultants define the effective divisor
\[
 T=\sum_i b_i^2\Div(\operatorname{disc}G_i)
 +2\sum_{i<i'}b_i b_{i'}\Div(\operatorname{res}(G_i,G_{i'}))
\]
on the base. A degree-one discriminant is a nonzero constant. The respective degrees are $d_i(d_i-1)$ and $d_id_{i'}$, so $\deg T\le k^2$. In a splitting DVR its order is
\[
 Q=\zeta(T)=2\sum_{i<i'}b_i b_{i'}c_{ii'},
\]
now indexing distinct roots with their assigned coefficients.

Choose the maximal contact $\rho$ with $v$ as in Lemma~\ref{lem:root}. Let $M(t)$ be the sum of the positive root coefficients in its contact cluster at level $t$, and let $N(t)$ be the sum of their squares. Then
\[
 \frac{v(\Theta^+)}e=\int_0^\rho M(t)\,dt,
 \qquad Q\ge\int_0^\rho\bigl(M(t)^2-N(t)\bigr)\,dt.
\]
Every $b_i\le1-\eta$, since the original signed pair is sub-$\eta$-lc. Therefore $N(t)\le(1-\eta)M(t)$ and
\[
 1+M(t)^2-N(t)-(1+\eta)M(t)
 \ge(M(t)-1)^2\ge0.
\]
The local smooth discrepancy divided by $e$ is at least $1+\rho$. Integration gives
\[
 (1+\eta)A(v,\PP^1_{\OO_\zeta},\Theta^+)
 \ge\eta A(v,\PP^1_{\OO_\zeta},0)+e-eQ.
\]
Adding the base canonical contribution $(a_0(\zeta)-1)e$ gives
\[
 (1+\eta)A(v,\PP^j,\Theta^+)
 \ge\eta A(v,\PP^j,0)+e a_0(\zeta)-e\zeta(T).
\]
The calculation uses only the individual coefficient bounds, not log canonicity of $(\PP^j,\Theta^+)$. Restoring the negative part increases the left-hand discrepancy. Lemma~\ref{lem:degree} bounds $\zeta(T)$ by $k^2a_0(\zeta)$ and proves \eqref{eq:comparison}. Our choice of $q$ puts the centre in the finite chart; the same inequality also holds at infinity because the positive boundary order there is zero.

The surface estimate over every base valuation gives $a_\Theta\ge\gamma=\eta^2/(6k)$. Apply the absolute coefficient-ideal construction and Lemma~\ref{lem:realize} to realize $a_\Theta(\zeta)$ by a signed sub-$\gamma$-lc divisor of positive degree at most $k^2$ on $\PP^{j-1}$. If $L_{j-1}(k^2,\gamma)$ is the comparison factor in dimension $j-1$, then
\begin{equation}\label{eq:retainquotient}
 e a_0(\zeta)
 \le L_{j-1}(k^2,\gamma)e a_\Theta(\zeta)
 \le L_{j-1}(k^2,\gamma)A(v,\PP^j,\Theta).
\end{equation}
Substituting into \eqref{eq:comparison} yields
\[
 L_j(k,\eta)\le\eta^{-1}
 \{2+k^2L_{j-1}(k^2,\eta^2/(6k))\}.
\]
With $L_1=\eta^{-1}$, the exponent recurrences are
\[
 p_j=1+2p_{j-1},\qquad b_j=2+2b_{j-1}+p_{j-1}.
\]
They give \eqref{eq:absolconstants}; since $k\ge1$ and $\eta\le1$, the coefficient recursion there absorbs the additive term two. In particular $A_2=8$ and $A_3=1730$.
\end{proof}

\begin{remark}
The decisive point is \eqref{eq:retainquotient}. Bounding the restriction index and the base discrepancy separately loses powers of $\eta$. The realization lemma instead permits induction directly on the quotient. No condition $a_\Theta(\zeta)\le1$ is used. Negative boundaries are indispensable for this formulation.
\end{remark}

\section{Finite morphisms and terminal Fano fibres}\label{sec:finite}

A terminal Fano variety need not be rational. We therefore transfer the projective-space estimates through a finite morphism. The transfer must preserve the minimum of the normalized discrepancies of all extensions of a valuation. The characteristic coefficients, rather than just the norm, provide this information.

\subsection{A finite collection of signed divisors}
\begin{proposition}\label{prop:transfer}
Let $\pi:T\to Y$ be finite surjective of degree $N$, where $T$ is normal and $Y=\PP^r$ or $\PP^r_R$. Suppose $K_T+B\simq0$ and $(T,B)$ is sub-$\eta$-lc. Set
\[
 H=K_T-\pi^*K_Y+B.
\]
Assume that $H$ is effective in the absolute case, or that its horizontal part is effective in the relative case. There are finitely many signed divisors $\Theta_j$ on $Y$ such that
\begin{align}
 (Y,\Theta_j)&\text{ is sub-}(\eta/N)\text{-lc},\label{eq:trans1}\\
 \deg\Theta_j^+&\le(r+1)N,
 \quad\text{or }\degh\Theta_j^+\le(r+1)N,\label{eq:trans2}\\
 \min_j A(\zeta,Y,\Theta_j)
 &=\min_{v_i|_{K(Y)}=e_i\zeta}
 \frac{A(v_i,T,B)}{e_i}\label{eq:trans3}
\end{align}
for every normalized divisorial valuation $\zeta$ of $Y$.
\end{proposition}
\begin{proof}
In the relative case choose a nonnegative rational $\ell$ for which $H+\ell\pi^*Y_0\ge0$; this is possible since there are finitely many vertical components and each occurs with positive multiplicity in $\pi^*Y_0$. Set $\ell=0$ absolutely. The divisor $H$ is $\Q$-Cartier and $H\simq-\pi^*K_Y$. The canonical line bundle of the local base $\Spec R$ is trivial. Consequently, for sufficiently divisible $m$ there is a section
\[
 s\in H^0\bigl(T,\pi^*\OO_Y((r+1)m)\bigr),\qquad
 \Div(s)=m(H+\ell\pi^*Y_0).
\]

In a local frame regard $s$ as an element of $K(T)$ and take the characteristic polynomial of multiplication by $s$ over $K(Y)$:
\[
 z^N+c_1z^{N-1}+\cdots+c_N.
\]
The coefficients are regular. Indeed, $s$ is integral over the normal base ring, the coefficients are integral over that ring, and they belong to its fraction field. They therefore lie in the base ring. Under a change of frame $c_j$ transforms as a section of $\OO_Y((r+1)mj)$. This argument does not require $\pi$ to be flat. Let $C_j=\Div(c_j)$ for a nonzero coefficient, and set $C_0=0$. Define
\begin{equation}\label{eq:thetaj}
 \Theta_j=\frac{C_N-C_j}{m(N-j)}-\ell Y_0
 \qquad(0\le j<N, c_j\ne0).
\end{equation}
Its positive degree is bounded by
\[
 \frac{(r+1)mN}{m(N-j)}\le(r+1)N.
\]
The same is true horizontally in the relative case. A larger choice of $\ell$ multiplies $c_j$ by a corresponding power of $u$; the changes cancel exactly in \eqref{eq:thetaj}.

Fix $\zeta$ and extend it to a splitting field without renormalizing. The orders of the characteristic roots are
\[
 m\left(\frac{v_i(H)}{e_i}+\ell\zeta(Y_0)\right),
\]
with the field-theoretic multiplicities. By Lemma~\ref{lem:polynomial}, the largest of these orders is recovered by \eqref{eq:maxroot}. Subtracting this largest order, divided by $m$, from $A(\zeta,Y,0)+\ell\zeta(Y_0)$ and using \eqref{eq:finite-ram} proves \eqref{eq:trans3}. For each normalized extension, $A(v_i,T,B)\ge\eta$ and $e_i\le N$. The minimum in \eqref{eq:trans3} is therefore at least $\eta/N$, proving \eqref{eq:trans1} for every member of the finite collection.
\end{proof}

\begin{remark}
The divisor realizing the minimum in \eqref{eq:trans3} can depend on $\zeta$, but the finite collection is fixed. A boundary formed from the norm alone would replace the largest root order by a sum or average, and would not give \eqref{eq:trans3}.
\end{remark}

\subsection{Uniform finite projections}
\begin{lemma}\label{lem:projections}
For each $r\ge1$ there are constants $V_r^{\mathrm{term}}>0$ and $N_r\in\Z_{>0}$ with the following properties. Every terminal Fano $r$-fold has anticanonical volume at most $V_r^{\mathrm{term}}$ and admits a finite morphism to $\PP^r$ of degree at most $N_r$. The finite-morphism assertion also holds over a function field over $\C$ if the variety is geometrically terminal Fano. One may take $N_1=2$.
\end{lemma}
\begin{proof}
Terminal Fano $r$-folds are $1$-lc and form a bounded family by \cite[Theorem~1.1]{Bir21}. Their volumes are bounded. The Cartier indices are uniformly bounded by \cite[Theorem~1.2]{HJ25}. Replacing the bound by a common multiple, choose $I_r$ for which $-I_rK_F$ is Cartier for every member. Effective base point freeness \cite{Kol93}, applied to this ample divisor, gives a uniform integer $\nu_r$ such that $L=-\nu_rK_F$ is Cartier, ample, and globally generated. Thus $L^r\le\nu_r^r V_r^{\mathrm{term}}$.

Choose $r+1$ sufficiently general sections of $L$ with no common zero. Such tuples exist: at a fixed point their simultaneous vanishing has codimension $r+1$ in the space of tuples, whereas the variety has dimension $r$. The resulting morphism to $\PP^r$ is finite, since its pullback of $\OO(1)$ is ample. Its degree is $L^r$.

The same geometric bounds apply over a function field over $\C$. Indeed, descend a projective presentation and a log resolution to a field finitely generated over $\Q$ and embed that field into $\C$. Geometric terminality, the Fano condition, the Cartier property of a fixed canonical multiple, generation, and intersection numbers are unchanged after extension between algebraically closed fields of characteristic zero. The complex case therefore supplies the same integers $I_r,\nu_r$ and the same degree bound on the geometric generic fibre. The line bundle $-\nu_rK_F$ is defined over the original field. Cartierness and global generation descend along a field extension, and global sections commute with that extension. Since the original field is infinite, the nonempty open set of suitable tuples has a rational point. This gives the finite morphism over that field. A smooth genus-zero curve has a globally generated anticanonical line bundle of degree two, proving the last assertion.
\end{proof}

\begin{proposition}[Terminal interpolation]\label{prop:terminalinterp}
There is $t_r>0$, depending only on $r$, such that for effective anticanonical divisors $B,D$ on a terminal Fano $r$-fold, with $(F,B)$ $\eta$-lc,
\[
 \bigl(F,(1-t)B+tD\bigr)\text{ is lc for }0\le t\le t_r\eta^{p_r}.
\]
One possible choice is
\begin{equation}\label{eq:tr}
 t_r=\left[1+(k_r-1)A_r k_r^{b_r}N_r^{p_r}\right]^{-1},
 \qquad k_r=(r+1)N_r.
\end{equation}
\end{proposition}
\begin{proof}
Take a finite morphism $\pi:F\to\PP^r$ of degree $N\le N_r$. The ramification divisor $R_\pi$ is effective, so $H_B=R_\pi+B$ and $H_D=R_\pi+D$ are effective. For a normalized valuation $v$ restricting to $e\zeta$, Proposition~\ref{prop:transfer} and Theorem~\ref{thm:absolute} give
\begin{equation}\label{eq:terminalgood}
 eA(\zeta,\PP^r,0)\le L A(v,F,B),\qquad
 L=A_r((r+1)N)^{b_r}(\eta/N)^{-p_r}.
\end{equation}
To see the direction of the inequality, choose the signed divisor minimizing the target discrepancy at $\zeta$. Its value is at most $A(v,F,B)/e$, and the absolute comparison applies to that divisor.

For a Cartier multiple of $H_D$, its norm is an effective divisor of degree $(r+1)mN$ on $\PP^r$. All characteristic root orders are nonnegative, so the order $mv(H_D)/e$ of any one root is at most the order of the norm. Lemma~\ref{lem:degree} yields
\[
 v(H_D)/e\le(r+1)N A(\zeta,\PP^r,0).
\]
By \eqref{eq:finite-ram} and \eqref{eq:terminalgood},
\[
 A(v,F,D)\ge-((r+1)N-1)L A(v,F,B).
\]
Discrepancy is affine in the boundary. Hence interpolation is lc for $t\le[1+((r+1)N-1)L]^{-1}$. Since $N\le N_r$ and $\eta\le1$, \eqref{eq:tr} is a valid uniform choice.
\end{proof}

\subsection{A discriminant bound over every base valuation}
\begin{proposition}\label{prop:discriminant}
Let $f:W\to Z$ be a contraction whose geometric generic fibre is a terminal Fano $r$-fold. Suppose
\[
 B\ge0,\qquad K_W+B\simq0,\qquad (W,B)\text{ is }\eta\text{-lc}.
\]
For the function in \eqref{eq:discrdef}, one has
\begin{equation}\label{eq:discriminant}
 a_B(\zeta)\ge d_r\eta^{2^r},\qquad
 d_r=\frac{1}{6^{2^r-1}(r+1)^{r2^{r-1}}N_r^{(r+2)2^{r-1}}},
\end{equation}
for every normalized divisorial valuation of $K(Z)$.
\end{proposition}
\begin{proof}
Represent $\zeta$ on a smooth base model and localize at its geometric DVR $R$. Lemma~\ref{lem:projections} gives a finite projection of the generic fibre to $\PP^r_K$, of degree $N\le N_r$. Normalize $\PP^r_R$ in $K(W)$ and denote the resulting model by $T$. This normalization is finite: $R$ is essentially of finite type over $\C$, so the usual finiteness of normalization in a finite function-field extension applies. Its generic fibre is the given Fano fibre, because that fibre is already normal and finite over $\PP^r_K$.

Use a rational function representing a multiple of $K_W+B\simq0$ to define the crepant signed boundary $B_T$ on $T$. All its discrepancies are the original ones, so the pair is sub-$\eta$-lc. Its horizontal part is the effective generic boundary, though the vertical part can be negative. Horizontal ramification is effective as well. Proposition~\ref{prop:transfer} applies.

For a valuation $v$ on $K(W)$, write its restriction to $K(\PP^r_R)$ as $e\xi$. Properness supplies centres on the finite model. Choose the minimizing member of the transferred collection at $\xi$. By \eqref{eq:trans3} and Theorem~\ref{thm:relative},
\[
 \frac{A(v,W,B)}{v(u)}
 \ge\frac{(\eta/N)^{2^r}}
 {6^{2^r-1}((r+1)N)^{r2^{r-1}}}.
\]
Indeed, $v(u)=e\xi(u)$, so the restriction index cancels. The exponent of $N$ is $2^r+r2^{r-1}=(r+2)2^{r-1}$. Taking the infimum over all $v$ lying above the original base valuation gives \eqref{eq:discriminant}. The argument works for exceptional base valuations as well as divisors on $Z$ itself.
\end{proof}

\section{Anticanonical interpolation in arbitrary dimension}\label{sec:interpolation}

\subsection{Approximating nef moduli parts}
The ordinary interpolation theorem must be applied to generalized pairs on the base. The following approximation is formulated for a fixed valuation; its parameter remains uniform for all valuations.

\begin{lemma}\label{lem:moduli}
Let $Z$ be $\Q$-factorial, projective and of Fano type. For $i=0,1$, suppose
\[
 B_i\ge0,\qquad K_Z+B_i+M_{i,Z}\simq0,
\]
where $\bM_i$ are $\Q$-b-Cartier and b-nef. Assume the generalized pair with index zero is generalized $\delta$-lc, $0<\delta\le1$; no singularity assumption is imposed for index one. Fix a normalized divisorial valuation $\zeta$. There are effective $\Q$-divisors $\Gamma_{i,j}\simq-K_Z$ such that
\[
 (Z,\Gamma_{0,j})\text{ is }\delta/2\text{-lc},
 \qquad A(\zeta,Z,\Gamma_{i,j})\longrightarrow
 A(\zeta,Z,B_i+\bM_i).
\]
\end{lemma}
\begin{proof}
Choose a klt log Fano boundary $\Delta$ on $Z$, and set $A=-(K_Z+\Delta)$. Take a common smooth model $p:U\to Z$ representing $\zeta$, carrying nef traces $M'_i$, and resolving the relevant boundaries. Write
\[
 K_U+B'_i+M'_i=p^*(K_Z+B_i+M_{i,Z}),\qquad
 K_U+\Delta'=p^*(K_Z+\Delta).
\]
We may dominate the model by a projective resolution obtained by blowups. There is an effective exceptional divisor $E$ with $-E$ relatively ample. For sufficiently small rational $e>0$, $p^*A-eE$ is ample. If $p$ is the identity, take $E=0$.

For rational $0<\lambda\le1/4$ and $0<s<\lambda e$, the class
\begin{equation}\label{eq:moduliclass}
 (1-\lambda)M'_i+\lambda p^*A-sE
\end{equation}
is ample. It is a sum of the nef divisor $(1-\lambda)M'_i$, the ample divisor $(s/e)(p^*A-eE)$, and the nef divisor $(\lambda-s/e)p^*A$. Choose a general effective $\Q$-representative $G_i$ of this class, with arbitrarily small coefficients. It may be taken transverse to the fixed support and not containing the prime representing $\zeta$. Put
\[
 \Gamma_i=(1-\lambda)B_i+\lambda\Delta+p_*G_i\ge0.
\]
Its crepant trace on $U$ is exactly
\begin{equation}\label{eq:modulitrace}
 (1-\lambda)B'_i+\lambda\Delta'+sE+G_i.
\end{equation}
To justify exactness, the log canonical divisor of \eqref{eq:modulitrace} is $\Q$-linearly trivial by \eqref{eq:moduliclass}. Represent a divisible multiple by a principal divisor, push it down to $Z$, and pull it back. The same rational function proves equality of the divisors, as in Lemma~\ref{lem:book}(2). It also proves $\Gamma_i\simq-K_Z$.

Every coefficient of $(1-\lambda)B'_0+\lambda\Delta'$ is at most $1-(1-\lambda)\delta\le1-3\delta/4$. Choose $s$ small enough that each coefficient of $sE$ is at most $\delta/8$, and choose the coefficients of $G_0$ at most $\delta/8$. All old coefficients stay at most $1-\delta/2$, the new ones satisfy the same bound, and the support is simple normal crossing. Lemma~\ref{lem:snc} gives the global $\delta/2$-lc conclusion.

Since $G_i$ avoids the prime representing $\zeta$, \eqref{eq:modulitrace} gives
\[
 A(\zeta,Z,\Gamma_i)
 =(1-\lambda)A(\zeta,Z,B_i+\bM_i)
 +\lambda A(\zeta,Z,\Delta)-s\coef_\zeta E.
\]
Let $\lambda$ and $s$ tend to zero rationally with $0<s<\lambda e$. This gives the required sequences.
\end{proof}

\begin{corollary}\label{cor:generalized}
Assume Theorem~\ref{thm:interp} in dimension $s=\dim Z$. In the setting of Lemma~\ref{lem:moduli}, write $a_i(\zeta)=A(\zeta,Z,B_i+\bM_i)$. Then
\begin{equation}\label{eq:geninterp}
 (1-\sigma)a_0(\zeta)+\sigma a_1(\zeta)\ge0,
 \qquad \sigma=c_s(\delta/2)^{p_s},
\end{equation}
for every normalized divisorial valuation of $Z$.
\end{corollary}
\begin{proof}
Fix $\zeta$, apply ordinary interpolation to the two divisors in each approximation from Lemma~\ref{lem:moduli}, and pass to the limit. The same number $\sigma$ works for every approximation and for every $\zeta$.
\end{proof}

\begin{remark}
The representatives in Lemma~\ref{lem:moduli} are allowed to depend on the prescribed valuation. Corollary~\ref{cor:generalized} first fixes the universal interpolation parameter, then chooses the representatives and takes the limit at that valuation. It does not assert simultaneous realization of the generalized discrepancies by a single ordinary boundary. Only b-nefness is used.
\end{remark}

\subsection{The dimension induction}
\begin{proof}[Proof of Theorem~\ref{thm:interp}]
We first work on $\Q$-factorial varieties. Decrease all constants to be at most $1/4$. In dimension one the variety is $\PP^1$. An effective anticanonical divisor has degree two, so every coefficient of $D$ is at most two. Every coefficient of $B$ is at most $1-\eta$. Thus $t\le\eta/4$ suffices, and we take $c_1=1/4$.

Assume the theorem in smaller dimensions, and write $\eta$ for the current singularity parameter. Take the terminalization in Lemma~\ref{lem:terminalization}. Its two effective anticanonical boundaries are crepant to $(X,B)$ and $(X,D)$. Run a $K_Y$-MMP. The underlying varieties remain terminal, and Fano type gives termination. Since $-K_Y$ is big, the output is a terminal Mori fibre space $f:W\to Z$. Both pairs remain effective and crepant throughout, by Lemma~\ref{lem:book}. The valuation statement consequently continues to test divisors contracted during the MMP.

If $Z$ is a point, apply Proposition~\ref{prop:terminalinterp}. Otherwise put
\[
 r=\dim W-\dim Z,\qquad s=\dim Z,
 \qquad r,s>0.
\]
The base is $\Q$-factorial and of Fano type. The geometric generic fibre is terminal Fano. Restrict a log resolution over a smooth open subset of the base; the good boundary on the geometric generic fibre is $\eta$-lc. The induction hypothesis applies to its two boundaries. Choose a rational number
\begin{equation}\label{eq:choose-t}
 c_r\eta^{p_r}/4\le t\le c_r\eta^{p_r}/2,
 \qquad C_t=(1-t)B_W+tD_W.
\end{equation}
The generic-fibre pair with boundary $C_t$ is klt: interpolation is lc at $c_r\eta^{p_r}$, and moving at most halfway from the $\eta$-lc boundary towards that lc boundary keeps all generic-fibre discrepancies at least $\eta/2$.

The use of the induction over the generic field is justified by spreading out a fixed log resolution and any divisorial valuation being tested. If an asserted inequality failed on the geometric generic fibre, it would fail on general complex fibres after a finite base change. The lower-dimensional theorem on those fibres excludes this. No boundedness statement with parameter $\eta$ is used here.

The effective boundaries $B_W$ and $C_t$ define klt-trivial fibrations. The rank-one condition and the effectiveness of the discriminant traces were checked in Section~\ref{sec:prelim}. The canonical bundle formula gives
\[
 K_Z+B_{Z,i}+M_{i,Z}\simq0\qquad (i=0,t)
\]
with $\Q$-b-Cartier b-nef moduli parts. The good generalized pair is $\delta$-lc for
\[
 \delta=d_r\eta^{2^r}
\]
by Proposition~\ref{prop:discriminant}. Apply Corollary~\ref{cor:generalized} on the base with $\sigma=c_s(\delta/2)^{p_s}$.

For a valuation $v$ nontrivial on the base, write $v|_{K(Z)}=e\zeta$ by Lemma~\ref{lem:restriction}. From the definition of the discriminant quotients and affine linearity,
\begin{align*}
 A\bigl(v,W,(1-\sigma)B_W+\sigma C_t\bigr)
 &=(1-\sigma)A(v,W,B_W)+\sigma A(v,W,C_t)\\
 &\ge e\{(1-\sigma)a_0(\zeta)+\sigma a_t(\zeta)\}\ge0.
\end{align*}
Valuations trivial on the base field are valuations of the generic fibre and have already been controlled there. Thus the final pair is lc for all divisorial valuations. The coefficient of $D_W$ in this pair is $t\sigma$, and
\[
 t\sigma\ge\frac{c_rc_s}{4}
 \left(\frac{d_r}{2}\right)^{p_s}
 \eta^{p_r+2^rp_s}
 =\frac{c_rc_s}{4}\left(\frac{d_r}{2}\right)^{p_s}
 \eta^{2^{r+s}-1}.
\]
It is therefore enough to choose recursively
\begin{equation}\label{eq:cn}
 c_n=\min\left\{\frac14,\ t_n,
 \min_{\substack{r+s=n\\r,s>0}}
 \frac{c_rc_s}{4}\left(\frac{d_r}{2}\right)^{p_s}\right\}.
\end{equation}
Lc at the constructed parameter implies lc at every smaller parameter by convexity with the good boundary. Crepancy transfers the assertion back to $X$.

Finally, for a non-$\Q$-factorial $X$, take a small projective $\Q$-factorialization. Since $K_X$ is $\Q$-Cartier, both anticanonical divisors are $\Q$-Cartier. Their crepant transforms on the small modification are effective. That modification is of Fano type by the weak log Fano perturbation used in Lemma~\ref{lem:terminalization}. Apply the proved result there and descend crepantly. Sharpness is proved in Proposition~\ref{prop:examples}.
\end{proof}

\subsection{The form needed for volume estimates}
\begin{corollary}\label{cor:logfanointerp}
Let $(F,\Delta)$ be an $\eta$-lc log Fano pair of dimension $r$. For any effective $D\simq-K_F$, the pair
\[
 (F,(1-t)\Delta+tD)
\]
is klt whenever $0\le t\le\mu_r(\eta)$, where
\begin{equation}\label{eq:mr}
 \mu_r(\eta)=m_r\eta^{p_r},\qquad
 m_r=\frac{c_r}{2^{p_r+1}}.
\end{equation}
In particular $m_1=1/16$.
\end{corollary}
\begin{proof}
Complete $\Delta$ to an effective anticanonical boundary $B$ which is $\eta/2$-lc. Theorem~\ref{thm:interp} gives an lc pair at $t_0=c_r(\eta/2)^{p_r}$. For $t\le t_0/2$, its convex combination with the good pair has discrepancies at least $\eta/4$. Replacing $B$ by the smaller divisor $\Delta$ only increases discrepancies.
\end{proof}

\begin{corollary}\label{cor:alpha}
Let $(X,\Delta)$ be an $\eta$-lc log Fano $n$-fold. For every effective $H\simq-(K_X+\Delta)$,
\[
 \lct(X,\Delta;H)\ge c_n(\eta/2)^{p_n}.
\]
\end{corollary}
\begin{proof}
Complete $\Delta$ to $B=\Delta+A\simq-K_X$ with $(X,B)$ $\eta/2$-lc, and put $D=\Delta+H$. Interpolation at $t=c_n(\eta/2)^{p_n}$ gives the lc boundary $\Delta+(1-t)A+tH$. Remove the effective divisor $(1-t)A$.
\end{proof}

\section{Reduction to a bounded polarized base}\label{sec:base}

The constants in the final volume estimate must depend on dimension alone. In particular, a base degree obtained by applying BAB with the varying parameter $\eta$ would not suffice. The following reduction uses terminality of the unmarked total space to start a separate sequence of fixed-parameter bounds.

We recall the two external results needed here. Birkar's base-singularity theorem \cite[Corollary~1.2]{BirFib} states that, for fixed relative dimension $d$ and $\alpha>0$, a Fano fibration $V\to T$ with $V$ $\alpha$-lc and $K_T$ $\Q$-Cartier has a $\beta$-lc base, for some $\beta>0$ depending only on $d,\alpha$. A Mori base of a $\Q$-factorial variety is $\Q$-factorial, so the $\Q$-Cartier hypothesis is satisfied in all its applications below.

The lifting result \cite[Lemma~2.11]{MM} says that if $V\to T$ is a contraction between $\Q$-factorial varieties of Fano type and $T\dashrightarrow T'$ is a birational contraction to a $\Q$-factorial variety, then there is a commutative diagram
\[
\begin{array}{ccc}
 V&\dashrightarrow&V'\\
 \downarrow&&\downarrow\\
 T&\dashrightarrow&T'
\end{array}
\]
in which $V\dashrightarrow V'$ is a birational contraction, is an isomorphism over the generic point of $T$, and $V'\to T'$ is a contraction of Fano type with $V'$ $\Q$-factorial. In the absolute situation $V'$ is again of Fano type. These properties are the ones used; no preservation of terminality of $V'$ is required.

\begin{proposition}\label{prop:boundedbase}
Fix $n$. There are constants $V_n^{\mathrm{term}}>0$ and $D_n\ge1$, depending only on $n$, with the following property. If $X$ is of $\eta$-Fano type and has dimension $n$, then either
\[
 \vol(-K_X)\le V_n^{\mathrm{term}},
\]
or there are a $\Q$-factorial variety $X'$, an effective $\Q$-divisor $\Delta'$, and a contraction $f:X'\to Z$ such that
\begin{enumerate}
\item $\vol(-K_X)\le\vol(-K_{X'})$;
\item $(X',\Delta')$ is $\eta/4$-lc and $-(K_{X'}+\Delta')$ is ample;
\item $1\le b=\dim Z<n$, and $Z$ has a very ample Cartier divisor $H$ with $H^b\le D_n$.
\end{enumerate}
\end{proposition}
\begin{proof}
As before, we complete the given log Fano boundary to an effective anticanonical boundary $B$ which is $\eta/2$-lc. Take a terminalization $Y\to X$ and run a $K_Y$-MMP, ending with a terminal Mori fibre space $W\to Z_0$. The anticanonical boundary stays effective, crepant and $\eta/2$-lc. Lemmas~\ref{lem:book} and~\ref{lem:terminalization} show that the volumes do not decrease. If $Z_0$ is a point, then $W$ is terminal Fano and the first alternative follows from Lemma~\ref{lem:projections}.

Suppose $\dim Z_0>0$. The base is $\Q$-factorial and of Fano type. Apply Birkar's theorem to $W\to Z_0$ with the fixed parameter $1/2$. Since $W$ is terminal, this gives a number $\delta_0>0$ depending only on the dimensions for which $Z_0$ is $\delta_0$-lc. This application does not involve the good boundary or its parameter $\eta$.

Run a $K_{Z_0}$-MMP. Its birational part $Z_0\dashrightarrow Z_0'$ is a birational contraction, and discrepancies of the unmarked varieties do not decrease. Its final step is a Mori fibre space $Z_0'\to Z_1$. Lift the birational part to the total space by the stated lifting lemma. Compose the resulting contraction to $Z_0'$ with $Z_0'\to Z_1$.

If $Z_1$ is a point, use $Z_0'$ as the final base. It is a $\delta_0$-lc Fano variety. Otherwise apply Birkar's theorem to $Z_0'\to Z_1$ with the fixed parameter $\delta_0$. This gives a $\delta_1$-lc base, where $\delta_1$ depends only on the preceding dimensions. Repeat the same procedure on $Z_1$. At each repetition the dimension of the base strictly decreases. There are fewer than $n$ repetitions, and only finitely many possible sequences of dimensions. Every parameter obtained in this process thus depends only on $n$ and on such a finite sequence.

The procedure ends with a positive-dimensional Fano base with a dimension-only lower discrepancy bound. BAB makes the collection of these final bases bounded. A bounded projective family admits a finite collection of relative very ample polarizations; their fibre degrees are bounded. Taking the maximum over the possible dimensions and sequences gives $D_n$. Compositions still have connected fibres, so the final morphism $X'\to Z$ is a contraction.

Every lifted map of total spaces was a birational contraction. Their anticanonical volumes do not decrease, and the final total space carries the effective $\eta/2$-lc boundary $B'\simq-K_{X'}$ crepant to $B$. Choose any effective klt log Fano boundary $\Delta_0$ on the Fano type variety $X'$ and set
\[
 \Delta'=\tfrac12B'+\tfrac12\Delta_0.
\]
Its discrepancies are at least $\eta/4$, and
\[
 -(K_{X'}+\Delta')\simq-\tfrac12(K_{X'}+\Delta_0)
\]
is ample. This proves all three properties.
\end{proof}

\begin{remark}\label{rem:fixed}
The initial total space is terminal, but the final total space and its general fibre need not be terminal Fano. Terminal boundedness is used at the first base step and in the point-base alternative only. Subsequent applications of the base-singularity theorem are to the lower-dimensional base Mori fibre spaces. The varying parameter is carried separately by the crepant anticanonical boundary. This separation prevents an unspecified function of $\eta$ from entering $D_n$.
\end{remark}

\section{The volume estimate and the main theorem}\label{sec:volume}

We next prove a volume estimate over a polarized base. It extends the curve-base connectedness argument in \cite[Theorem~6.1]{Jiang21} to a flag on a higher-dimensional base. The local construction of an isolated non-klt fibre is needed to make connectedness effective at each hyperplane step.

\subsection{Subtracting a free divisor}
\begin{lemma}\label{lem:sections}
Let $Y$ be a normal projective $d$-fold and $L$ a big $\Q$-Cartier divisor. Let $G$ be a normal general member of a base point free Cartier system, and suppose $L|_G$ is big. If
\[
 \vol_Y(L)>dq\vol_G(L|_G),\qquad q>0,
\]
then $L-qG$ is big. The same assertion holds when $G$ is a general fibre divisor of a contraction to a smooth curve.
\end{lemma}
\begin{proof}
First let $q$ be rational and take divisible $m$. Repeated restriction sequences give
\[
 h^0(Y,mL-mqG)\ge h^0(Y,mL)
 -\sum_{j=0}^{mq-1}h^0(G,(mL-jG)|_G).
\]
Because $\OO_G(G)$ is globally generated, multiplication by a nonzero section of $\OO_G(jG)$ bounds every summand by $h^0(G,mL|_G)$. If $\OO_G(G)$ is trivial, the same statement is immediate. Divide by $m^d/d!$ and use the asymptotics for volumes to obtain
\begin{equation}\label{eq:subtract}
 \vol_Y(L-qG)\ge\vol_Y(L)-dq\vol_G(L|_G)>0.
\end{equation}
For a general fibre over a smooth curve, the normal line bundle is trivial, so the proof is identical. If $q$ is real, choose a slightly larger rational number retaining the strict inequality and use monotonicity of volume.
\end{proof}

\subsection{Isolating a whole fibre}
\begin{lemma}\label{lem:isolate}
Let $f:Y\to Z$ be projective and surjective, with $Z$ normal of dimension $k\ge2$. Let $C\ge0$, with $K_Y+C$ $\Q$-Cartier, and suppose $(Y,C)$ is klt over the generic point of $Z$. If $H$ is very ample on $Z$, then for a sufficiently general smooth point $z\in Z$ there is an effective $\Q$-divisor $T\simq kH$ such that, over a neighbourhood of $z$,
\[
 \Nklt(Y,C+f^*T)=f^{-1}(z).
\]
\end{lemma}
\begin{proof}
Take a log resolution $g:V\to Y$, with $K_V+C_V=g^*(K_Y+C)$. Shrink to a smooth open subset $U\subset Z$ so that $V\to U$ and all strata of the resolved boundary are smooth, and all remaining coefficients are less than one. To do this, remove the images of nondominating strata and the critical loci of the finitely many dominating strata. Properness and generic smoothness show that the removed set is proper.

Choose $z\in U$, an integer $M>k$, and $M$ sufficiently general members $H_j\in|H|$ through $z$. Their tangent hyperplanes are in general position. Put
\[
 T=\frac{k}{M}\sum_{j=1}^M H_j.
\]
On the blowup of $z$, the exceptional divisor has crepant coefficient
\[
 M\frac{k}{M}-(k-1)=1.
\]
Every strict transform has coefficient $k/M<1$. Near the exceptional divisor these strict transforms and that divisor have simple normal crossing support: at most $k-1$ general tangent hyperplanes meet on $\PP^{k-1}$. Away from $z$ the general members have simple normal crossings after shrinking $U$. Thus the base pair is lc near $z$ and its non-klt locus there is precisely $\{z\}$.

Base change this blowup to the smooth morphism $V\to U$. The resulting space is smooth, the exceptional divisor has coefficient one, and the horizontal boundary retains its coefficients less than one. Smoothness of every stratum ensures simple normal crossings with the pulled-back arrangement. Consequently its non-klt locus maps exactly onto $(f\circ g)^{-1}(z)$. The exceptional divisor surjects onto that fibre. Crepancy and proper surjectivity of $g$ now give the asserted equality on $Y$, including all singular points of the fibre.
\end{proof}

\subsection{A flag estimate}
\begin{proposition}\label{prop:flag}
Let $(X,\Delta)$ be an $\eta$-lc log Fano $n$-fold and let $f:X\to Z$ be a contraction with $1\le b=\dim Z<n$. Set $r=n-b$, let $H$ be very ample on $Z$, and put $d=H^b$. Assume a sufficiently general fibre $F$ satisfies
\[
 \vol(-K_F)\le M,
\]
and that $(F,(1-t)\Delta_F+tD_F)$ is klt for every effective $D_F\simq-K_F$ and every $0\le t<\mu$, where $0<\mu\le1$. Then
\begin{equation}\label{eq:flag}
 \vol(-K_X)\le
 \frac{n!}{r!}\bigl((b-1)d+2\bigr)(b+1)^{b-1}M\mu^{-b}.
\end{equation}
\end{proposition}
\begin{proof}
We may choose a general flag cut by $b-1$ members of $|H|$ on $Z$, and take its inverse images:
\[
 Z=Z_0\supset Z_1\supset\cdots\supset Z_{b-1},
 \qquad X=Y_0\supset Y_1\supset\cdots\supset Y_{b-1}.
\]
The base sections are normal, and the last one is a smooth curve. On the total space the successive free systems have image dimension at least two, so general members are irreducible. Bertini on a log resolution and adjunction show that the $Y_i$ are normal and klt. In particular the canonical restriction formulas below are valid. Put
\[
 L_i=-K_X|_{Y_i},\quad G_i=f_i^*(H|_{Z_i}),\quad
 \Delta_i=\Delta|_{Y_i},\quad
 A_i=-(K_X+\Delta)|_{Y_i},\quad U_i=\vol_{Y_i}(L_i).
\]
Here $A_i$ is ample, $L_i=A_i+\Delta_i$ is big, and
\begin{equation}\label{eq:flagadj}
 K_{Y_i}=-L_i+iG_i.
\end{equation}

Fix $0\le i\le b-2$, and suppose
\[
 U_i>(n-i)(b+1)\mu^{-1}U_{i+1}.
\]
Choose rational $q>(b+1)/\mu$ retaining $U_i>(n-i)qU_{i+1}$. Lemma~\ref{lem:sections} gives an effective $D_i\simq L_i-qG_i$. Set
\[
 t=(b+1)/q<\mu,\qquad C_i=(1-t)\Delta_i+tD_i.
\]
On a sufficiently general fibre, $D_i|_F\simq-K_F$, so the pair is klt there by hypothesis. Equivalently it is klt over a dense open subset of $Z_i$. Put $k=b-i\ge2$. Lemma~\ref{lem:isolate} gives $T\simq kH|_{Z_i}$ isolating a sufficiently general fibre $F_z$ as its entire local non-klt locus. Choose one further general hyperplane section $H_\infty$ of $Z_i$ avoiding $z$ and give it coefficient one.

The non-klt locus of
\[
 (Y_i,C_i+f_i^*T+f_i^*H_\infty)
\]
has the whole fibre $F_z$ as a nonempty open-and-closed part. It has a second nonempty part over $H_\infty$: its pullback components have positive integral multiplicities and the other boundary is effective. These two parts are disjoint. On the other hand, \eqref{eq:flagadj} gives
\begin{align*}
 -(K_{Y_i}+C_i+f_i^*T+f_i^*H_\infty)
 &\simq(1-t)A_i+(tq-i-k-1)G_i\\
 &=(1-t)A_i,
\end{align*}
which is ample. The connectedness theorem \cite[Theorem~5.48]{KM98} contradicts this disconnection. Hence
\begin{equation}\label{eq:flagsteps}
 U_i\le(n-i)(b+1)\mu^{-1}U_{i+1}
 \qquad(0\le i\le b-2).
\end{equation}

It remains to consider $Y=Y_{b-1}$ over the smooth curve $Z_{b-1}$. Its general fibre $F$ is Cartier, and $G_{b-1}\equiv dF$. Thus
\[
 K_Y\equiv-L_{b-1}+(b-1)dF.
\]
If
\[
 U_{b-1}>(r+1)((b-1)d+2)M\mu^{-1},
\]
choose rational $q>((b-1)d+2)/\mu$ with
$U_{b-1}>(r+1)q\vol(-K_F)$. The fibre version of Lemma~\ref{lem:sections} gives $D\simq L_{b-1}-qF$, $D\ge0$. Set
\[
 t=((b-1)d+2)/q<\mu,\qquad C=(1-t)\Delta_{b-1}+tD.
\]
The pair is klt over a dense open subset of the curve. Add two distinct sufficiently general fibres $F_1,F_2$ with coefficient one. Its non-klt locus has finite image on the curve, and that image contains both points. It is therefore disconnected. But
\[
 -(K_Y+C+F_1+F_2)\equiv(1-t)A_{b-1}
\]
is ample, again contradicting connectedness. Consequently
\begin{equation}\label{eq:flagcurve}
 U_{b-1}\le(r+1)((b-1)d+2)M\mu^{-1}.
\end{equation}
Multiplying \eqref{eq:flagsteps} and \eqref{eq:flagcurve} proves \eqref{eq:flag}. For $b=1$, the hyperplane steps are absent and the curve argument is the entire proof.
\end{proof}

\begin{remark}
The small coefficients in Lemma~\ref{lem:isolate} are necessary here. If one instead used coefficient one hyperplanes through $z$, each hyperplane would be non-klt and could connect the fibre to $H_\infty$. Isolating the whole fibre first gives the open-and-closed part needed in the connectedness argument.
\end{remark}

\subsection{Proof of the volume theorem}
\begin{proof}[Proof of the upper bound in Theorem~\ref{thm:main}]
Let $a_j=2^j-j-1$ and recall $p_j=2^j-1$. In dimension one take $C_1=2$. Assume the theorem in smaller dimensions and apply Proposition~\ref{prop:boundedbase} to an $\eta$-Fano type $n$-fold. The terminal-Fano alternative is bounded by $V_n^{\mathrm{term}}$ and can be absorbed into $C_n$.

In the other alternative there is an $\eta/4$-lc log Fano pair on $X'$ over a base of dimension $b$ and polarized degree at most $D_n$. Write $r=n-b$. Restricting a log resolution over a smooth open subset of the base preserves the $\eta/4$ discrepancy bound on the general fibre. It is therefore of $\eta/4$-Fano type. By induction and Corollary~\ref{cor:logfanointerp}, the hypotheses of Proposition~\ref{prop:flag} hold with
\[
 M=C_r(\eta/4)^{-a_r},\qquad
 \mu=m_r(\eta/4)^{p_r}.
\]
All remaining factors in \eqref{eq:flag} depend only on dimensions. Its exponent of $\eta^{-1}$ is
\begin{equation}\label{eq:exponents}
 a_r+bp_r=(b+1)2^r-n-1\le2^{r+b}-n-1=a_n,
\end{equation}
because $b+1\le2^b$. Equality holds for $b=1$. Since $0<\eta\le1$, a single constant absorbs all cases. More precisely, a valid recursive choice is
\begin{equation}\label{eq:Cn}
\begin{split}
 C_n=\max\biggl\{V_n^{\mathrm{term}},\ 
 \max_{1\le r<n}\frac{n!}{r!}
 &\bigl((n-r-1)D_n+2\bigr)(n-r+1)^{n-r-1}\\[-2pt]
 &\cdot C_r m_r^{-(n-r)}4^{a_r+(n-r)p_r}\biggr\}.
\end{split}
\end{equation}
This proves the volume assertion for $\eta$-Fano type varieties.

If $X$ is an $\eta$-lc weak Fano variety, the base point free theorem makes a multiple of $-K_X$ free. Its anticanonical morphism is birational, since $-K_X$ is big. On the anticanonical model $Y$, the descended ample $\Q$-Cartier divisor equals $-K_Y$, by pushing forward on the isomorphism locus. The morphism is crepant, $K_X=f^*K_Y$, and $Y$ remains $\eta$-lc. Its anticanonical volume equals that of $X$. Applying the Fano case gives the stated extension. The optimality assertion is proved in Section~\ref{sec:examples}.
\end{proof}

\begin{remark}[Real boundaries]\label{rem:real}
Suppose the log Fano boundary in the hypothesis is an effective real divisor. On a small $\Q$-factorialization its negative log canonical divisor is nef and big. Write it as $A\sim_\R H+G$ with $H$ ample and $G\ge0$. Add a sufficiently small rational multiple of $G$ to the boundary, preserving $\eta/2$-lc and making the negative log canonical divisor ample. On a log resolution of the resulting fixed support, crepant coefficients depend affinely on the boundary coefficients. A sufficiently close effective rational approximation preserves $\eta/4$-lc, by Lemma~\ref{lem:snc}, and ampleness, by openness of the ample cone. The rational-boundary theorem applies with $\eta/4$. Small modifications preserve anticanonical volume, so only the dimension-only constant changes.
\end{remark}

\subsection{Dimension four and the constants}
For fourfolds, the three possible fibre dimensions in \eqref{eq:exponents} give
\[
\begin{array}{c|c|c|c|c}
 r&b&a_r&p_r&a_r+bp_r\\\hline
 3&1&4&7&11\\
 2&2&1&3&7\\
 1&3&0&1&3
\end{array}
\]
Thus the relevant interpolation exponent for the fourfold volume bound is $p_3=7$, not the separate fourfold interpolation exponent $p_4=15$. Jiang--Zou's threefold theorem \cite[Theorem~1.1]{JZ24} may replace the lower-dimensional volume induction in the first row. For example, using their strict parameter $\eta/4$ gives $C_3=819200$ for $0<\eta\le1$.

With $C_1=2$ and $m_1=1/16$, formula~\eqref{eq:Cn} specializes to
\begin{equation}\label{eq:Cfour}
 C_4=\max\left\{
 \begin{array}{l}
 V_4^{\mathrm{term}},\\
 33554432\,C_3m_3^{-1},\\
 589824\,C_2(D_4+2)m_2^{-2},\\
 402653184\,(D_4+1)
 \end{array}\right\}.
\end{equation}
This displays the dependence on the intermediate constants, but is not an explicit value of $C_4$. In particular, $D_4$ in Proposition~\ref{prop:boundedbase} is obtained from a bounded family after applying the base-singularity theorem. The cited results as used here do not specify an explicit value for this degree. Some other inputs can be made explicit in low dimension, but doing so alone would not make \eqref{eq:Cfour} numerical. We therefore leave $C_4$ dimension dependent, as in the general theorem.

\section{Examples and further discussion}\label{sec:examples}

\subsection{Optimality in every dimension}
We describe the family underlying \cite[Example~6.3]{Amb16}, numbered Example~5.3 in arXiv:1411.2770v1. The discrepancy-index calculation below checks the singularity parameter for all valuations, and not only for the invariant divisors.

\begin{proposition}\label{prop:examples}
For every $n\ge2$ and integer $h\ge3$, there is a well-formed weighted projective Fano variety $X_{n,h}$ of dimension $n$, with $\Q$-factorial singularities and Picard number one, such that $X_{n,h}$ is $1/h$-lc and
\begin{equation}\label{eq:sharpvol}
 \vol(-K_{X_{n,h}})\sim h^{2^n-n-1}\qquad(h\longrightarrow\infty).
\end{equation}
There are also effective anticanonical $\Q$-divisors $B_h,D_h$ on $X_{n,h}$, with $(X_{n,h},B_h)$ $1/h$-lc, for which lc of $(X_{n,h},(1-t)B_h+tD_h)$ implies
\begin{equation}\label{eq:sharpinterp}
 t\le\frac{h}{u_{n+1}},\qquad
 \frac{h}{u_{n+1}}\sim h^{-(2^n-1)}.
\end{equation}
Here $u_j$ is the integer sequence in the proof. The symbol $\sim$ in \eqref{eq:sharpvol} and \eqref{eq:sharpinterp} means that the ratio tends to one.
\end{proposition}
\begin{proof}
Set
\begin{equation}\label{eq:sylvester}
 u_1=h,\qquad u_{j+1}=u_j(1+u_j),\qquad
 q_j=1+u_j,\qquad U=u_{n+1},\qquad L=U/h.
\end{equation}
The recursion gives $L=\prod_{j=1}^n q_j$. The integers $q_j$ are pairwise coprime: every later $u_j$ is divisible by every earlier $q_i$, hence $q_j\equiv1\pmod{q_i}$ for $i<j$. Define
\[
 a_j=L/q_j,\qquad X_{n,h}=\PP(1,a_1,\ldots,a_n),
\]
and let $S$ be the coordinate divisor corresponding to the weight one. The gcd of $a_1,\ldots,a_n$ is one, since the factors $q_j$ are pairwise coprime. The presence of the weight one gives the other well-formedness conditions. Thus the usual class-group and canonical-divisor formulas for weighted projective spaces apply; see \cite{Dolg82}. In particular $S$ is an ample $\Q$-Cartier divisor, the variety is $\Q$-factorial of Picard number one, and
\[
 -K_X\simq QS,\qquad Q=1+\sum_{j=1}^n a_j.
\]

Telescoping the recursion in \eqref{eq:sylvester} gives
\[
 \sum_{j=1}^n\frac1{q_j}
 =\sum_{j=1}^n\left(\frac1{u_j}-\frac1{u_{j+1}}\right)
 =\frac1h-\frac1U.
\]
Consequently
\begin{equation}\label{eq:Q}
 Q=1-\frac1h+\frac{U}{h^2},\qquad
 h\bigl(K_X+(1-1/h)S\bigr)\sim-LS.
\end{equation}
Every weight divides $L$, so $LS$ is Cartier. The second divisor in \eqref{eq:Q} is therefore an integral Cartier divisor.

The toric pair $(X,(1-1/h)S)$ is klt. To check this, its log discrepancy function has positive values on every ray generator: $1/h$ on the ray of $S$ and one on all the others. It is linear on each cone and positive on every nonzero point of that cone. On a toric resolution all crepant coefficients are therefore less than one, and the boundary has simple normal crossing support. Lemma~\ref{lem:snc}, with the minimum of the finitely many positive discrepancies on that resolution, proves klt for all valuations. Since the log canonical divisor has Cartier index dividing $h$, every log discrepancy on every resolution lies in $h^{-1}\Z$. Positivity thus makes all log discrepancies at least $1/h$. The divisor $S$ attains $1/h$. Removing the boundary shows that the underlying Fano variety is $1/h$-lc.

For clarity, the weighted intersection formula can be read from the Hilbert series
\[
 \frac{1}{(1-z)\prod_{j=1}^n(1-z^{a_j})}.
\]
The leading coefficient of its Hilbert quasi-polynomial is $1/(n!\prod_j a_j)$; hence $S^n=1/\prod_j a_j$. Also $\prod_j a_j=L^{n-1}$. Using \eqref{eq:Q}, we obtain
\begin{equation}\label{eq:volumeexample}
 \vol(-K_X)=\frac{Q^n}{L^{n-1}}
 =\frac{U}{h^{n+1}}
 \left(1+\frac{h^2-h}{U}\right)^n.
\end{equation}
Induction in \eqref{eq:sylvester} gives $u_j\sim h^{2^{j-1}}$, and in particular $U\sim h^{2^n}$. For $n\ge2$ the parenthesis in \eqref{eq:volumeexample} tends to one. This proves \eqref{eq:sharpvol}.

Finally, $-(K_X+(1-1/h)S)\simq(U/h^2)S$ is ample. By Lemma~\ref{lem:book}(1), complete this marked boundary to $B_h\simq-K_X$ while retaining $1/h$-lc and the coefficient $1-1/h$ along $S$. Take $D_h=QS$. Lc interpolation implies that the coefficient along $S$ is at most one:
\[
 1-1/h+t\{Q-(1-1/h)\}\le1.
\]
Substitute \eqref{eq:Q} to get $t\le h/U$. This proves \eqref{eq:sharpinterp} and the optimality assertions in Theorems~\ref{thm:main} and~\ref{thm:interp}.
\end{proof}

The same examples show that the exponent in Corollary~\ref{cor:alpha} is optimal. Indeed, with $\Delta=(1-1/h)S$ and $H=(U/h^2)S$, one has $H\simq-(K_X+\Delta)$ and $\lct(X,\Delta;H)\le h/U$.

\subsection{The log Fano assumption is necessary}
It is not enough to assume that $X$ itself is $\eta$-lc and that $X$ is of Fano type without a parameter. For example, let $X=\mathbb F_m$ be a Hirzebruch surface, $m\ge3$, with negative section $S$ and fibre $F$. Thus
\[
 S^2=-m,\qquad S\cdot F=1,\qquad F^2=0,\qquad
 -K_X=2S+(m+2)F.
\]
The boundary $\Delta_m=(1-1/m)S$ is klt and
\[
 -(K_X+\Delta_m)=(1+1/m)S+(m+2)F
\]
is ample, since its intersections with $S$ and $F$ are one and $1+1/m$. All $\mathbb F_m$ are consequently of Fano type, and they are smooth, hence $1$-lc. On the other hand, the Zariski decomposition is
\[
 -K_X=(1+2/m)(S+mF)+(1-2/m)S,
\]
where the first summand is nef and orthogonal to the negative section. Therefore
\[
 \vol(-K_X)=\frac{(m+2)^2}{m}\longrightarrow\infty.
\]
This example explains the formulation of $\eta$-Fano type used in the main theorem. Products with projective spaces give the analogous distinction in higher dimension.

\subsection{Additional structure and further questions}
The estimate over a bounded base also yields improvements when the fibre volume is already bounded independently of the singularity parameter.

\begin{corollary}\label{cor:curveimproved}
Let $n\ge2$, and let $(X,\Delta)$ be an $\eta$-lc log Fano $n$-fold admitting a contraction to a smooth curve whose general fibre is terminal Fano. There is $E_n>0$, depending only on $n$, such that
\[
 \vol(-K_X)\le E_n\eta^{-(2^{n-1}-1)}.
\]
In particular, the exponent is seven for fourfolds under this additional hypothesis.
\end{corollary}
\begin{proof}
Use Proposition~\ref{prop:flag} with $b=1$, $r=n-1$, the bound $M=V_{n-1}^{\mathrm{term}}$, and $\mu=m_{n-1}\eta^{p_{n-1}}$. The base degree disappears. One may take $E_n=2nV_{n-1}^{\mathrm{term}}m_{n-1}^{-1}$.
\end{proof}

The extra hypothesis cannot be imposed on the final morphism in Proposition~\ref{prop:boundedbase}. Lifted birational contractions need not preserve terminality, and composing base contractions need not preserve the Fano condition on the resulting generic fibre. This is why the general argument still uses the variable-parameter volume induction.

The optimal numerical constant in \eqref{eq:main} remains open here. Formula~\eqref{eq:volumeexample} has leading coefficient one with the indicated singularity parameters, while the upper-bound argument loses substantial constants in projection, finite transfer and the change of base. A numerical value of $C_4$ would also require an effective version of the bounded polarization used in Proposition~\ref{prop:boundedbase}.

There is a related coefficient question. Theorem~\ref{thm:interp} gives an upper bound of order $\eta^{-p_n}$ for the coefficient of any prime in an effective anticanonical divisor. If the good boundary already contains the same prime with coefficient $1-\eta$, its coefficient in $D$ is at most
\[
 1-\eta+c_n^{-1}\eta^{-(p_n-1)}.
\]
This follows by testing the interpolation inequality at that prime. The marked examples above have coefficients of order $\eta^{-(2^n-2)}$, so this marked exponent is sharp. Whether the same exponent holds without the marked-boundary assumption is a separate problem; it is not needed for the volume theorem.

Finally, Theorem~\ref{thm:relative} supplies an exponent $2^r$ for the relative discrepancy estimate with bounded horizontal degree. We do not claim its optimality in every relative dimension. Better estimates for particular generic fibres or restricted degenerations could sharpen the corresponding interpolation constants or the bounds under additional geometric assumptions, even though the unrestricted volume exponent is fixed by Proposition~\ref{prop:examples}.

\end{document}